\documentclass[11pt]{amsart}
\usepackage[a4paper, total={6in, 8in}]{geometry}
\usepackage{enumerate}
\usepackage{amsmath}
\usepackage{amssymb,latexsym}
\usepackage{amsthm}
\usepackage{color}
\usepackage{cancel}
\usepackage{graphicx}
\usepackage{cite}
\usepackage{changes}
\usepackage[hidelinks]{hyperref}
\usepackage{comment}
\usepackage{amscd}
\usepackage{mathtools}

\DeclareMathOperator{\C}{\mathcal{C}}

\usepackage[foot]{amsaddr}

\DeclareMathOperator{\rk}{rk}

\newtheorem{theorem}{Theorem}[section]

\newtheorem{lemma}[theorem]{Lemma}
\newtheorem{corollary}[theorem]{Corollary}
\newtheorem{definition}[theorem]{Definition}
\newtheorem{proposition}[theorem]{Proposition}

\newtheorem{remark}[theorem]{Remark}

\newtheorem{construct}[theorem]{Construction}

\newcommand{\fqn}{\mathbb{F}_{q^n}}

\newcommand{\F}{{\mathbb F}}

\newcommand{\fq}{{\mathbb F}_{q}}

\newcommand{\la}{\langle}
\newcommand{\ra}{\rangle}

\newcommand{\PG}{\mathrm{PG}}

\newcommand\qbin[3]{\left[\begin{matrix} #1 \\ #2 \end{matrix} \right]_{#3}}

\newcommand{\classcode}{\mathcal{C}[n,k,d]_q}
\newcommand{\classsystem}{\mathcal{P}[n,k,d]_q}

\newcommand{\vspan}[1]{\left \langle #1 \right \rangle}
 \newcommand{\set}[1]{ \left \{ #1 \right \} }

\newcommand{\sam}[1]{\textcolor{purple}{#1}}

\title[Cones from maximum $h$-scattered linear sets and a stability result]{Cones from maximum $h$-scattered linear sets and a stability result for cylinders from hyperovals}
\author[S. Adriaensen]{Sam Adriaensen}
\email{sam.adriaensen@vub.be}
\author[J. Mannaert]{Jonathan Mannaert}
\address{Sam Adriaensen and Jonathan Mannaert, \textnormal{Department of Mathematics and Data Science, Vrije Universiteit Brussel, Brussels, Belgium}}
\email{jonathan.mannaert@vub.be}
\author[P. Santonastaso]{Paolo Santonastaso}
\email{paolo.santonastaso@unicampania.it}
\author[F. Zullo]{Ferdinando Zullo}
\address{Paolo Santonastaso and Ferdinando Zullo, \textnormal{Dipartimento di Matematica e Fisica, Universit\`a degli Studi della Campania ``Luigi Vanvitelli'', Viale Lincoln, 5, I--\,81100 Caserta, Italy}}
\email{ferdinando.zullo@unicampania.it}
\date{}

\begin{document}

\maketitle

\begin{abstract}
    This paper mainly focuses on cones whose basis is a maximum $h$-scattered linear set. We start by investigating the intersection sizes of such cones with the hyperplanes. Then we analyze two constructions of point sets with few intersection sizes with the hyperplanes. In particular, the second one extends the construction of translation KM-arcs in projective spaces, having as part at infinity a cone with basis a maximum $h$-scattered linear set. As an instance of the second construction we obtain cylinders with a hyperoval as basis, which we call \emph{hypercylinders}, for which we are able to provide a stability result. The main motivation for these problems is related to the connections with both Hamming and rank distance codes. Indeed, we are able to construct codes with few weights and to provide a stability result for the codes associated with hypercylinders.
\end{abstract}

\bigskip
{\it AMS subject classification (2020):} 51E20; 51E21; 94B05.

\bigskip
{\it Keywords:} linear set;  scattered linear set; Hamming metric code; rank metric code.

\section{Introduction}

Scattered linear sets (and more generally scattered spaces) were defined and investigated for the first time in 2000 by Blokhuis and Lavrauw in \cite{blokhuis2000scattered}. Since their introduction, scattered linear sets have found fertile ground in Galois geometries and in coding theory, see e.g.\ \cite{lavrauw2016scattered,polverino2020connections}.
In this paper we will mainly focus our attention on maximum $h$-scattered linear sets.
Let $V$ be an $r$-dimensional vector space over $\fqn$ and let $\Lambda=\PG(V,\F_{q^n})=\PG(r-1,q^n)$ be the associated projective space.
If $U$ is a $k$-dimensional $\fq$-subspace of $V$, then the set of points
\[ L_U=\{\la {u} \ra_{\mathbb{F}_{q^n}} : { u}\in U\setminus \{{ 0} \}\}\subseteq \Lambda \]
is said to be an $\fq$-\textbf{linear set of rank $k$}.
An important notion related to linear sets is \textbf{the weight of a subspace} $\Omega$ with respect to $L_U$, which is a measure of how much of the linear set is contained in $\Omega$. 
If all the $(h-1)$-dimensional projective subspaces have weight at most $h$, then $L_U$ is said to be \textbf{$h$-scattered}; see \cite{csajbok2021generalising}. When $h=1$ this exactly coincides with the notion introduced by Blokhuis and Lavrauw in \cite{blokhuis2000scattered} and when $h=r-1$ it coincides with the notion introduced by Sheekey and Van de Voorde in \cite{sheekeyVdV}. 
The rank of an $h$-scattered linear set in $\PG(r-1,q^n)$ is bounded by $rn/(h+1)$ and an $h$-scattered linear set with this rank is called a \textbf{properly maximum $h$-scattered linear set}. 
For these linear sets the intersection numbers with respect to the hyperplanes are known (see \cite{blokhuis2000scattered,csajbok2021generalising,zini2021scattered}) and interestingly they take exactly $h+1$ distinct values.
In this paper we will first study cones having as basis a properly maximum $h$-scattered $\fq$-linear set in a complementary space to the vertex. 
The possible intersection sizes of such a set with a hyperplane can be easily derived from the intersection numbers of the basis with respect to hyperplanes.
Then we exploit two constructions of point sets in $\PG(r-1,q^n)$ which arise from cones of properly maximum $h$-scattered linear sets. More precisely, let $L_U$ be a cone with basis a properly maximum $h$-scattered linear set contained in a hyperplane $\pi_{\infty}$ of $\PG(r,q^n)$ and let $P=\langle v \rangle_{\fqn}\in \PG(r,q^n)\setminus \pi_{\infty}$. Then we can consider the following two point sets:
\begin{enumerate}
    \item $\mathcal{B}=L_{U'}$, where $U'=U\oplus \langle v \rangle_{\fq}$;
    \item $\mathcal{K}=(\pi_{\infty}\setminus L_U)\cup (\mathcal{B} \setminus \pi_{\infty})=(\pi_{\infty}\setminus \mathcal B)\cup (\mathcal B \setminus \pi_{\infty})$.
\end{enumerate}
The second construction can be seen as a generalization of the construction of \emph{translation KM-arcs}, which are point sets in $\PG(2,2^n)$ of the projective plane that can be all obtained by the second construction replacing the cone with a special type of linear set of rank $n$ on the line at infinity (known as a \emph{club}) with $q=2$, see \cite[Theorem 2.1]{deboeck2016linear}.
For both of the constructions we determine the possible intersection sizes with the hyperplanes, which are strongly related to the intersection sizes with the hyperplanes of the chosen properly maximum $h$-scattered linear set, obtaining \emph{sets with few intersection numbers} with respect to the hyperplanes; see \cite{de2011constructions}.
As a special instance of Construction (2) we obtain the \textbf{hypercylinder}, that is a cone with as basis a hyperoval, and as vertex a subspace of codimension 3, where the vertex is then deleted. We prove a stability result for hypercylinders obtaining that when considering a point set with size close to the size of a hypercylinder and with at most three possible intersection sizes with the hyperplanes, then it necessarily is a hypercylinder. The main tool regards some results on KM-arcs proved by Korchm\'aros and Mazzocca in \cite{korchmaros1990on} and two results of Calkin, Key and de Resmini in \cite{calkin1999minimumweight} on even sets.

The main motivation for studying these point sets is certainly related to coding theory. Indeed, using the well-known correspondence between projective systems (or systems) and Hamming metric codes (respectively rank metric codes), we are able to provide constructions of codes with \emph{few weights} in both Hamming and rank distances and to provide a stability result for the codes associated with the hypercylinder (in the Hamming metric). 
The latter codes deserve attention as they present only three nonzero weights.

The paper is organized as follows. In Section \ref{sec:prel} we discuss some preliminaries that will be useful later on. These focus mostly on linear sets, $h$-scattered linear sets, even sets and KM-arcs. In Section \ref{sec:aux} we prove some results on linear sets which are frequently used in the paper and regard the size of certain families of linear sets.
Section \ref{sec:constr} is mainly devoted first to the study of cones with basis a properly maximum $h$-scattered linear set and then to the determination of the intersection sizes of the hyperplanes with both Constructions (1) and (2).
As an instance of Construction (2) we obtain the hypercylinders. In Section \ref{sec:stab} we provide a stability result for hypercylinders, making use of combinatorial techniques and some combinatorial results on KM-arcs. In Section \ref{sec:codes}, after describing the connections between Hamming/rank metric codes and projective systems/systems, we are able to construct codes with few weights and to provide a stability result for those codes arising from hypercylinders. This is indeed a consequence of the results obtained in the previous sections. Finally, we conclude the paper with Section \ref{sec:final} in which we summarize our results and list some open problems/questions.

\section{Preliminaries}\label{sec:prel}

We consider the projective space $\PG(r,q)$, with $r\geq 2$ and $q$ a prime power, unless otherwise stated.

\begin{proposition}[{\cite[Theorem 3.1.1]{hirschfeld1979}}] \label{prop:numbersubapcescontaining}
The number of $k$-spaces in $\PG(r,q)$ containing a fixed $h$-space is 
\[
\qbin{r-h}{k-h}{q}=\frac{(q^{r-h}-1)(q^{r-h-1}-1)\cdots (q^{r-k+1}-1)}{(q^{k-h}-1)(q^{k-h-1}-1)\cdots (q-1)}
\]
\end{proposition}

We will denote by $[k+1]_q$ the number of points of a $k$-space of $\PG(r,q)$, i.e.
\[
[k]_q=\frac{q^{k}-1}{q-1}.
\]

In the paper we will frequently use the following notions.

\begin{definition}
Let $\mathcal K$ be a set of points in $\mathrm{PG}(r,q)$. Suppose that there exist $s$ positive integers $m_1 < \cdots < m_s$ such that for every $k$-space $\sigma$, with $k\geq 1$ fixed,
$$\vert \sigma \cap {\mathcal K}\vert \in \{m_1, \ldots, m_s\}.$$
Then we say that $\mathcal{K}$ is of \textbf{type} $\{m_1, \ldots, m_s\}_k$.
In case $k=1$ and all the $m_i$'s are even, $\mathcal K$ is called an \textbf{even set}.
If each of the integers $m_i$ occurs as the size of the intersection of $\mathcal K$ with a $k$-space of $\mathrm{PG}(r, q)$, we say that $\mathcal{K}$ is of  \textbf{type} $(m_1, \ldots, m_s)_k$, and we call the $m_i$ the \textbf{intersection numbers}. 
\end{definition}

We note that if $\mathcal K$ is an even set in $\PG(r,q)$, then either $q$ is odd and $\mathcal K \in \{\emptyset, \PG(r,q)\}$, or $q$ is even and $\mathcal K$ intersects every subspace of dimension at least 1 in an even number of points.

\subsection{KM-arcs}  

Ovals and hyperovals are well studied objects in finite geometries.

\begin{definition}
Suppose that $\mathcal{O}$ is a set of points in $\PG(2,q)$ such that no three points are collinear. Then $\mathcal{O}$ is called an \textbf{oval} of $\PG(2,q)$ if it has $q+1$ points, and a \textbf{hyperoval} if  it has $q+2$ points. 
\end{definition}

It can be seen that every line intersects an oval in $0$, $1$, or $2$ points and  every line intersects a hyperoval in $0$ or $2$ points, and all of these cases occure. This makes an oval a set of type $(0,1,2)_1$ and a hyperoval a set of type $(0,2)_1$.

The standard example of an oval is a \textbf{conic}.
Up to the action of $\mathrm{PGL}(3,q)$, there is a unique conic, namely the solutions over $\mathbb{F}_{q}^3$ to the equation $Y^2 = X Z$.
Moreover some classification results for ovals are known. One of which was proven by Segre in 1955.
\begin{theorem}[{\cite[Theorem 1]{segre_1955}}]
Suppose that $q$ is odd.
Then every oval in $\PG(2,q)$ is a conic.
Consequently, there are no hyperovals in $\PG(2,q)$.
\end{theorem}
However, for $q > 4$ even, there exist other examples of ovals; see e.g.\ \cite{Brown}.
Furthermore, each oval in $\PG(2,q)$, $q$ even, can be extended to a hyperoval. Therefore, hyperovals always exist in $\mathrm{PG}(2,q)$ with $q$ even.

Next, we define the KM-arcs, first introduced and investigated by Korchm\'aros and Mazzocca in \cite{korchmaros1990on}.

\begin{definition}
A \textbf{KM-arc of type $t$} in $\PG(2,q)$ is a set of $q+t$ points of $\PG(2,q)$ of type $(0,2,t)_1$. 
\end{definition}
It is immediately clear that KM-arcs are in fact generalizations of ovals and hyperovals, by setting $t=1$, respectively $t=2$.
However this generalization also has some interesting properties.

\begin{proposition}[{\cite[Proposition 2.1]{korchmaros1990on}}]
 \label{prop:kmarcs}
Let $\mathcal{K}$ be a KM-arc of type $t$ in $\PG(2,q)$, then:
\begin{itemize}
    \item $t$ divides $q$;
    \item if $1 < t < q$, then $q$ is even.
\end{itemize}
\end{proposition}

Recall that a KM-arc $\mathcal{K}$ in $\PG(2,q)$ is called a \textbf{translation} KM-arc if there exists a line $\ell$ of $\PG(2,q)$ such that the group of elations with axis $\ell$ and fixing $\mathcal{K}$ acts transitively on the points of $\mathcal{K}\setminus \ell$.

Finally, we list some known results on even sets in projective spaces, which have been stated and proved using a coding theoretical approach. The first regards a lower bound on the size of an even set with respect to the lines, whereas the second one is a characterization of those of minimum size.

\begin{theorem}[{\cite[Theorem 1]{calkin1999minimumweight}}]
 \label{Th:minimumsizecodeword}
Let $S$ be an even set in $\PG(r,q)$, q even. Then $\lvert S \rvert \geq q^{r-1}+2q^{r-2}$.
\end{theorem}

\begin{definition}
Let $\pi$ and $\sigma$ be complementary subspaces in $\PG(r,q)$, and take a set of points $S \subseteq \sigma$.
The \textbf{cone} $C$ with \textbf{vertex} $\pi$ and \textbf{basis} $S$ is the set of all points which lie on a line intersecting both $\pi$ and $S$, i.e.\
\[
 C = \bigcup_{P \in S} \la P,\pi \ra.
\]
We call $C \setminus \pi$ a \textbf{cylinder}.
If $\dim \pi = r-3$, and $S$ is a hyperoval in $\sigma$, we call $C\setminus \pi$ a \textbf{hypercylinder}.
\end{definition}

\begin{theorem}[{\cite[Proposition 3]{calkin1999minimumweight}}]
\label{th:classconehyperoval}
For $q \geq 4$ even, every point set in $\PG(3,q)$ of size $q^2+2q$ of type $(0,2,q)_1$ is a hypercylinder.
\end{theorem}

\begin{remark}
The theorem above also holds for $q=2$, because in that case, the complement of such a set is a set of size $[3]_2$ of type $(1,3)_1$, hence a hyperplane, and the complement of a hyperplane is a hypercylinder.
\end{remark}

\subsection{Linear sets}

Let $V$ be an $r$-dimensional vector space over $\fqn$ and let $\Lambda=\PG(V,\F_{q^n})=\PG(r-1,q^n)$.
If $U$ is a $k$-dimensional $\fq$-subspace of $V$, then the set of points
\[ L_U=\{\la { u} \ra_{\mathbb{F}_{q^n}} : { u}\in U\setminus \{{ 0} \}\}\subseteq \Lambda \]
is said to be an $\fq$-\textbf{linear set of rank $k$}.
Note that when we use the notation $L_U$ for an $ \fq$-linear set, we are formally considering both the set of points it defines and the underlying subspace $U$.

If $\Omega=\PG(W,\F_{q^n})$ is a subspace of $\PG(r-1,q^n)$, the intersection of $L_U$ with $\Omega$ is the $\F_q$-linear set $L_{U \cap W}$. 
We say that $\Omega$ has \textbf{weight} $i$ with respect to $L_U$, denoted as $w_{L_U}(\Omega) = i$, if the $\F_q$-linear set $L_{W \cap U}$ has rank $i$, i.e.\ $w_{L_U}(\Omega) = \dim_{\F_q}(U \cap W)$.
If $N_i$ denotes the number of points of $\Lambda$ having weight $i\in \{0,\ldots,k\}$  in $L_U$, the following relations hold:
\begin{equation}\label{eq:card}
    |L_U| \leq [k]_q,
\end{equation}
\begin{equation}\label{eq:pesicard}
    |L_U| =N_1+\ldots+N_k,
\end{equation}
\begin{equation}\label{eq:pesivett}
    N_1+N_2(q+1)+\ldots+N_k(q^{k-1}+\ldots+q+1)=[k]_q.
\end{equation}

Furthermore, $L_U$ is called \textbf{scattered} if it has the maximum number $[k]_q$  of points, or equivalently, if all points of $L_U$ have weight one.

Also, the following property concerning the weight of subspaces holds true.

\begin{proposition}[{\cite[Property 2.3]{polverino2010linear}}]
\label{prop:containingsubspace}
Let $L_U$ be an $\F_q$-linear set of $\PG(r-1,q^n)$ of rank $k$ and let $\Omega$ be an $s$-space of $\PG(r-1,q^n)$.
Then $\Omega \subseteq L_U$ if and only if $w_{L_U}(\Omega) \geq sn+1$.
\end{proposition}

We refer to \cite{lavrauw2015field} and \cite{polverino2010linear} for comprehensive references on linear sets.

In \cite{csajbok2021generalising}, the authors introduced a special family of scattered linear sets, named $h$-scattered linear sets; see \cite[Definition 1.1]{csajbok2021generalising}.

\begin{definition}
Let $h$ be a positive integer such that $1\leq h\leq r-1$.
An $\F_q$-linear set $L_U$ of $\Lambda$ is called $h$-\textbf{scattered} (or scattered w.r.t.\ the $(h-1)$-dimensional subspaces) if $\la L_U \ra=\Lambda$ and each $(h-1)$-dimensional $\F_{q^n}$-subspace $\Omega=\PG(W,\fqn)$ of $\Lambda$ has weight in $L_U$ at most $h$.
\end{definition}

The $1$-scattered linear sets are the scattered linear sets generating $\Lambda$. The same definition applied to $h=r$ describes the canonical subgeometries of $\Lambda$ (i.e.\ the copies of $\PG(r-1,q)$ embedded in $\PG(r-1,q^n)$). If $h=r-1$ and $\dim_{\F_q}(U)=n$, then $L_U$ is $h$-scattered exactly when $L_U$ is a scattered $\F_q$-linear set with respect to the hyperplanes, introduced in \cite[Definition 14]{sheekeyVdV}; see also \cite{lunardon2017mrd}.

Theorem \ref{th:bound} bounds the rank of an $h$-scattered linear set.

\begin{theorem}[{\cite[Theorem 2.3]{csajbok2021generalising}}]
\label{th:bound}
If $L_U$ is an $h$-scattered $\F_q$-linear set of rank $k$ in $\Lambda=\PG(r-1,q^n)$, then one of the following holds:
\begin{itemize}
\item $k=r$ and $L_U$ is a subgeometry $\PG(r-1,q)$ of $\Lambda$;
\item $k\leq\frac{rn}{h+1}$.
\end{itemize}
\end{theorem}

An $h$-scattered $\F_q$-linear set of maximum possible rank is said to be a \textbf{maximum $h$-scattered} $\F_q$-linear set.
An $h$-scattered $\F_q$-linear set of rank $\frac{rn}{h+1}$ is said to be a \textbf{properly maximum $h$-scattered} $\F_q$-linear set.
Theorem \ref{th:inter} bounds the weight of the hyperplanes with respect to a maximum $h$-scattered linear set.

\begin{theorem}[{\cite[Theorem 2.7]{csajbok2021generalising}}]
 \label{th:inter}
If $L_U$ is a properly maximum $h$-scattered $\F_q$-linear set of $\Lambda$, then for any hyperplane $\sigma$ of $\Lambda$ we have
\[
 \frac{rn}{h+1}-n\leq w_{L_{U}}(\sigma) \leq \frac{rn}{h+1}-n+h.
\]
\end{theorem}

Since a properly maximum $h$-scattered linear set is also scattered and using the correspondence between $h$-scattered linear sets and maximum rank distance codes (see \cite{lunardon2017mrd,1930-5346_2019_0_108,sheekey2016new,sheekeyVdV,zini2021scattered}), one can completely determine the pattern of intersections with the hyperplanes.

\begin{theorem}[{\cite[Theorem 7.1]{zini2021scattered}}]
\label{th:intersectionhyper}
Let $L_U$ be a properly maximum $h$-scattered $\F_q$-linear set in $\Lambda$.
For every $i\in\left\{0,\ldots,h\right\}$, the number of hyperplanes of weight $\frac{rn}{h+1}-n+i$ in $L_U$, that is the number of hyperplanes meeting $L_{U}$ in $\left[\frac{rn}{h+1}-n+i\right]_q$
points, is
\begin{equation}\label{eq:ti} 
t_i=\frac{1}{q^n-1}{\qbin{n}{i}{q} } \sum_{j=0}^{h-i} (-1)^{j}{\qbin{n-i}{j}{q} } q^{\binom{j}{2}}\left(q^\frac{rn(h-i-j+1)}{h+1}-1\right). 
\end{equation}
In particular, $t_i>0$ for every $i\in\left\{0,\ldots,h\right\}$.
\end{theorem}

Constructions of $h$-scattered $\F_q$-linear sets have been given in \cite{csajbok2021generalising} and in \cite{napolitano2021linear}.

There are known constructions of properly maximum $h$-scattered linear sets \sam{of rank $\frac{rn}{h+1}$?} in $\Lambda$ in the following cases:
\begin{itemize}
    \item[a)] $nr$ is even and $h=1$, see \cite{ball2000linear,bartoli2018maximum,blokhuis2000scattered,csajbok2017maximum};
    \item[b)] $h+1 \mid r$ and $n\geq h+1$, see \cite{csajbok2021generalising,napolitano2021linear};
    \item[c)] $nr'$ is even, $h=n-3$ and $r=r'(n-2)/2$, see \cite{csajbok2021generalising}.
\end{itemize}

More recently, in \cite{bartoli2021evasive} a construction was exhibited of a $1$-scattered linear set of rank $7$ in $\PG(2,q^5)$, where $q=p^h$ and $p \in \{2,3,5\}$.

\section{Auxiliary results}\label{sec:aux}

Now we will prove some results which will be fundamental for the rest of the paper.

\begin{lemma} \label{lem:extensionpoint}
Let $U$ be a $k$-dimensional $\F_q$-subspace of $V$ \sam{$= \fqn^r$} and let $v \in V$ be such that $\langle v \rangle_{\F_{q^n}} \notin L_{U} $. Let $U_1=U \oplus \langle v \rangle_{\F_q}$. Then any point in $L_{U_1} \setminus L_{U}$ has weight one in $L_{U_1}$. Moreover, if $\langle v \rangle_{\F_{q^n}} \cap \langle U \rangle_{\F_{q^n}}=\{0\}$, then $\lvert L_{U_1} \rvert = \lvert L_{U} \rvert+q^k$.
\end{lemma}
\begin{proof}
Suppose there exists a point $P=\langle w \rangle_{\F_{q^n}} \in L_{U_1} \setminus L_{U}$ such that $w_{L_{U_1}}(P) >1$. Let $T=U_1 \cap \langle w \rangle_{\F_{q^n}}$. So, $\dim_{\F_q}(T)=w_{L_{U_1}}(P) >1$.
Since $U$ and $T$ are both contained in $U_1$, and $\dim_{\F_q} (U) = \dim_{\F_q} (U_1) - 1$, Grassmann's identity implies that
\[
w_{L_{U}}(P)= \dim_{\F_q} (\langle w \rangle_{\F_{q^n}} \cap U) =
 \dim_{\F_q} (T \cap U) \geq -\dim_{\F_q}(U_1) + \dim_{\F_q} (U) + \dim_{\F_q} (T) \geq 1,
\]
hence $P \in L_{U}$, a contradiction.
Suppose now that $\langle v \rangle_{\F_{q^n}} \cap \langle U \rangle_{\F_{q^n}}=\{0\}$. 
Let $P=\langle u \rangle_{\F_{q^n}} \in L_{U}$, with $u \in U$, then $w_{L_{U}}(P)=w_{L_{U_1}}(P)$. Indeed, 
if $\rho u=u'+\alpha v$, for some $u' \in U, \alpha \in \fq^*$ and $\rho \in \fqn^*$, we have that $\alpha v=\rho u-u'\in U$ 
implying that $\langle v \rangle_{\F_{q^n}} \cap \langle U \rangle_{\F_{q^n}} \neq \{0\}$, a contradiction. So, we have that $w_{L_{U}}(P)=w_{L_{U_1}}(P)$ for any $P \in L_{U}$.  Let $N_1'$ be the number of points (of weight one with respect to $L_{U_1}$) in $L_{U_1} \setminus L_{U}$. By \eqref{eq:pesicard} and \eqref{eq:pesivett}, it follows 
\[
\begin{cases}
N_1'+\lvert L_{U} \rvert= \lvert L_{U_1} \rvert \\
N_1'+[k]_q=[k+1]_q
\end{cases}
\]
Then $N_1'=q^k$ and $\lvert L_{U_1} \rvert =\lvert L_{U} \rvert+q^k$.
\end{proof}

The above lemma may be generalized in the following way.

\begin{proposition} \label{prop:weightbyintersectioninfinity}
Let $U$ be an $\F_q$-subspace of $V \sam{= \fqn^r}$ such that $\langle U  \rangle_{\F_{q^n}} \ne V$ and let $v \in V$ be such that $\langle v \rangle_{\F_{q^n}} \cap \langle U \rangle_{\F_{q^n}}=\{0\}$ and let $U_1=U\oplus \langle v \rangle_{\fq}$. 
Let $\Omega=\PG(W,\F_{q^n})$ be a subspace of $\PG(r-1,q^n)=\PG(V,\F_{q^n})$.
If $w_{L_{U}}(\Omega)=j>0$ and $ \lvert \Omega \cap (L_{U_1} \setminus L_{U} )\rvert >0$, then $\lvert L_{U_1} \cap \Omega \rvert = \lvert L_{U} \cap \Omega \rvert +q^j$.
\end{proposition}
\begin{proof}
By assumptions, $\dim_{\fq}(W\cap U)=j>0$ and $\dim_{\fq}(W\cap U_1)\geq j+1$. By contradiction, suppose that $\dim_{\fq}(W\cap U_1)\geq j+2$, then there exist $w_1,w_2 \in W\cap U_1$ which are $\fq$-linearly independent and such that $\langle w_1,w_2\rangle_{\fq} \cap (W\cap U)=\{0\}$. Moreover, there exist $u_1,u_1' \in U$ and $\alpha,\beta \in \fq^*$ such that
\[ w_1=u_1+\alpha v\,\,\mbox{and}\,\, w_2=u_1'+\beta v, \]
from which we have 
\[ \beta w_1 -\alpha w_2=\beta u_1-\alpha u_1'\in U, \]
that is $\beta w_1 -\alpha w_2 \in U \cap \langle w_1,w_2\rangle_{\fq}=\{0\}$ and hence
\[ \beta w_1=\alpha w_2, \]
a contradiction. 
Therefore, $w_{L_{U_1}}(\Omega)=\dim_{\fq}(W\cap U_1)= j+1$ and the assertion follows now by applying the previous lemma.
\end{proof}

\section{Constructions of point sets in projective spaces with few intersection numbers}\label{sec:constr}

In this section we provide families of a point sets in $\PG(r,q^n)$ having few intersection numbers with respect to hyperplanes. We essentially investigate two families, in particular the second one may be seen as an extension of the translation KM-arcs in a projective space version. This construction extends the notion of hypercylinder.

\subsection{Properties of cones with basis a linear set}

\begin{construct} \label{con:genhyper}
Let $S_1$ be a $(d-1)$-space of $\PG(r-1,q^n)$ and let $S_2=\PG(W,\F_{q^n})$ be an $(r-d-1)$-space complementary to $S_1$ in $\PG(r-1,q^n)$. 
Let $L=L_{U_1}$ be a properly maximum $h$-scattered linear set in $S_1$. Let $\mathcal{D} \subseteq \PG(r-1,q^n)$ be the cone having as base $L$ and as vertex $S_2$, i.e.\ $\mathcal D$ is the union of all lines through a point of $L$ and a point of $S_2$.
Then we have that $\mathcal{D}$ coincides with the $\F_q$-linear set $L_U$ in $\PG(r-1,q^n)$ with $U=U_1 \oplus W$,
see Figure \ref{fig:cons1}.

\bigskip

We can now choose a suitable coordinatization $(x_0,...,x_{r-1})$ for $\PG(r-1,q^n)$ in such a way that:
\begin{itemize}
    \item The $(d-1)$-space $S_1=\{\langle (x_0,...,x_{r-1}) \rangle_{\F_{q^n}} \colon x_d=...=x_{r-1}=0\}$.
    \item The $(r-d-1)$-space $S_2=\{\langle (x_0,...,x_{r-1}) \rangle_{\F_{q^n}} \colon x_0=...=x_{d-1}=0\}$
\end{itemize}
Then we have that $\mathcal{D}=L_U$ is an $\F_q$-linear set of $\PG(r-1,q^n)$, with
\begin{equation} \label{eq:formU}
U:=\{(x_0,\ldots,x_{r-1})\in \fqn^r \colon (x_0,\ldots,x_{d-1}) \in U_1 \}.\end{equation}

\begin{figure}[h!]
    \centering
    \includegraphics[scale=0.5]{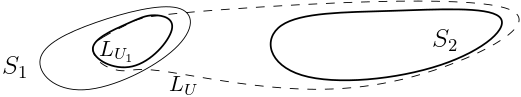}
    \caption{Construction \ref{con:genhyper}.}\label{fig:cons1}
\end{figure}
\end{construct}

\begin{lemma}\label{le:linset}
Let $L_U$ be the $\F_q$-linear set in $\PG(r-1,q^n)$ with $U$ as in \eqref{eq:formU}. Then $L_U$ has rank $\frac{dn}{h+1}+n(r-d)$, all points have weight $0,1$ or $n$ and $|L_U|=q^{n(r-d)}\left[\frac{dn}{h+1}\right]_q+\left[ r-d
 \right]_{q^n}$.
\end{lemma}
\begin{proof}
Since $L_{U_1}$ is a properly maximum $h$-scattered linear set in $\PG(d-1,q^n)$ and so it has rank $\frac{dn}{h+1}$, we have that $L_U$ has rank equal to $\frac{dn}{h+1}+n(r-d)$.
Secondly, consider the weight of the point $P=\langle (x_0,\ldots,x_{r-1}) \rangle_{\F_{q^n}} \in L_U$.
\begin{itemize}
    \item If $P \in S_2$, it automatically follows that $w_{L_U}(P)=n$.
    \item Now suppose that $P \in L_U \setminus S_2$, and let $x=(x_0, \dots, x_{r-1})$, $y=(y_0, \dots, y_{r-1})$ be two non-zero vectors in $U$ whose $\fqn$-span defines $P$.
    Then $y = \lambda x$, for some $\lambda \in \fqn$.
    Since $P \notin S_2$,
    \[
     (x_0, \dots, x_{d-1}) = \lambda (y_0, \dots, y_{d-1})
    \]
    are non-zero vectors of $U_1$.
    Since $U_1$ is scattered, $\lambda$ must be in $\fq$, which proves that $P$ has weight 1.
\end{itemize}
Thirdly, the size of $L_U$ can be determined easily using the fact that $L_U$ is a cone with vertex $S_2$ (which contains 
$[r-d]_{q^n}$
points) and as base $L_{U_1}$ (which contains 
$\left[\frac{dn}{h+1}\right]_q$ points).
\end{proof}

\subsection{The intersection sizes of the cone}\label{sec:intersect}

In this section, we calculate the intersection sizes of the set of points $L_U$ with hyperplanes in $\PG(r-1,q^n)$. 

We start with the following well-known property, which asserts that if we consider a cone with a basis for which we already know the intersection numbers with respect to the hyperplanes in its span, then we can completely determine the intersection numbers with respect the hyperplanes of the associated cone in the entire space.

\begin{proposition} \label{prop:intersectionconehyper}
Let $S_1$ and $S_2$ be a skew $(d-1)$-space and $(r-d-1)$-space in $\PG(r-1,q)$, respectively.
Consider a cone $\mathcal{D}$ with base $B \subset S_1$ and vertex $S_2$.
Suppose that $B$ is of type $(m_i)_{d-2}$ with respect to the hyperplanes of $S_1$. Then each hyperplane $\pi$ satisfies one of the following:
\begin{enumerate}
    \item $S_2 \subset \pi$ and \( |\pi \cap \mathcal{D}| = [r-d]_{q} + m_i q^{r-d} \) for some $i$,
    \item or $S_2 \not \subset \pi$, and \( |\pi \cap \mathcal{D}| =[r-d-1]_{q}+ |B| q^{r-d-1} \).
\end{enumerate}
\end{proposition}

We can now apply the previous proposition to determine the intersection pattern between the hyperplanes and any cone with basis a properly maximum $h$-scattered linear set.

\begin{corollary} \label{lem:intersectioninfinity}
Let $L_U$ be the $\F_q$-linear set in $\PG(r-1,q^n)$ with $U$ as in \eqref{eq:formU}. Consider a hyperplane $\pi$ of $\PG(r-1,q^n)$. Let $L_T=L_U \cap \pi$. Then one of the two cases holds:
\begin{enumerate}
    \item $L_T$ has rank $\gamma_i+n(r-d)$ where $\gamma_i:=\frac{dn}{h+1}-n+i$ and 
    $|L_T|=q^{n(r-d)}\left[\gamma_i\right]_q+[r-d]_{q^n}$, for $i\in \{0,\ldots, h\}$;
    \item $L_T$ has rank $n(r-d-1)+\frac{dn}{h+1}$ and $\lvert L_T \rvert=q^{n(r-d-1)}\left[\frac{dn}{h+1} \right]_q+[r-d-1]_{q^n}$.
    
\end{enumerate}
\end{corollary}

\begin{proof}
    The $\F_q$-linear set $L_U$ is the cone having as base $B=L_{U_1} \subseteq S_1$ and as a vertex a skew space $S_2=\PG(r-d-1,q^n)$ to $S_1$. By Theorem $\ref{th:inter}$, we know that $B$ is of type $\left([\gamma_i]_q\right)_i$, where $\gamma_i=\frac{dn}{h+1}-n+i$  for $i \in \{0,\ldots,h\}$. Then by Proposition \ref{prop:intersectionconehyper} the assertion follows.
\end{proof}


\subsection{First general construction}\label{sec:constr1}

In the following construction we will consider a cone with an $h$-scattered linear set in $\pi_\infty$.
The classical way is to consider the linear set defined by the chosen linear set in $\pi_\infty$ and an extra affine point. See Figure \ref{fig:cons2}.

\begin{figure}[ht]
    \centering
    \includegraphics[scale=0.4]{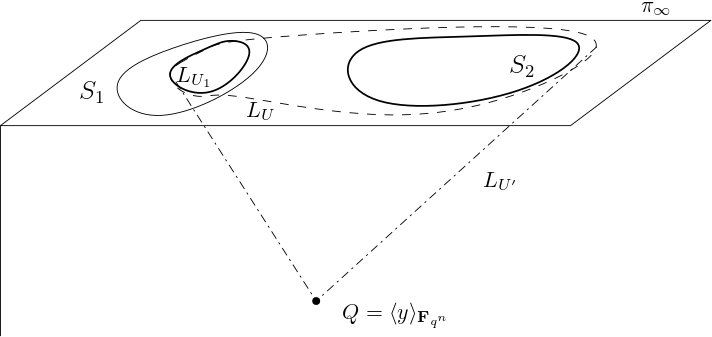}
    \caption{Theorem  \ref{th:intersectsizes}.}\label{fig:cons2}
\end{figure}

\begin{theorem}\label{th:intersectsizes}
Let $\mathcal{D}=L_{U} \subseteq \pi_{\infty}:=\PG(r-1,q^n)$ be as in Construction \ref{con:genhyper} and embedded in $\PG(r,q^n)$. Let $Q=\langle y \rangle_{\F_{q^n}} \in \PG(r,q^n) \setminus \pi_{\infty}$. Let $U'=U \oplus \langle y \rangle_{\F_{q}}$.
Then the point set $\mathcal{B}=L_{U'} \subseteq \PG(r,q^n)$ has size
\[
q^{n(r-d)}\left[\frac{dn}{h+1}+1\right]_q+[r-d]_{q^n}
\]
and $\mathcal{B}$ is of type
\begin{center}
\scalebox{0.9}{%
$\left(q^{n(r-d)}\left[\frac{dn}{h+1}\right]_q+[r-d]_{q^n}, q^{n(r-d-1)}\left[\frac{dn}{h+1}+1\right]_q+[r-d-1]_{q^n},
q^{n(r-d)}\left([\gamma_i]_q +\beta_{i,j}\right) +[r-d]_{q^n}
\right)_{r-1}$,
}
\end{center}

where for $i\in\{0,...,h\}$ and $j \in \{0,1\}$, $\gamma_i:=\frac{dn}{h+1}-n+i$, and 
\[\beta_{i,j} = \begin{cases}
0 & \text{if } j=0 \text{ and } i \geq 1\\
q^{\gamma_i} & \text{otherwise}.
\end{cases}
\]
\end{theorem}
\begin{proof}
Consider $\pi_\infty:=\PG(r-1,q^n)$ as a subspace of $\PG(r,q^n)$. 
We first make the convention to consider $\pi_\infty$ as the hyperplane at infinity, so we can easily describe the points of $\PG(r,q^n)\setminus \pi_\infty$ as affine points.
Since $L_{U}$ has rank $\frac{dn}{h+1}+n(r-d)$, it follows from Lemma \ref{lem:extensionpoint} and Lemma \ref{le:linset} that 
\begin{equation*}
    \begin{split}
        \lvert\mathcal{B}\rvert&=\lvert L_U\rvert+q^{\frac{dn}{h+1}+n(r-d)}\\
        &=q^{n(r-d)}\left[\frac{dn}{h+1}\right]_q+[r-d]_{q^n}+q^{\frac{dn}{h+1}+n(r-d)}\\
        &=q^{n(r-d)}\left[\frac{dn}{h+1}+1\right]_q+[r-d]_{q^n}
    \end{split}
\end{equation*}
Secondly, we are interested in the intersection sizes with respect to the hyperplanes. This boils down to two cases: let $\pi$ be an arbitrary $(r-1)$-space of $\PG(r,q^n)$ then either $\pi=\pi_{\infty}$ or $\pi\ne\pi_{\infty}$. Suppose that $\pi=\pi_\infty$: in this case clearly
    \begin{equation*}
        \begin{split}
            \lvert \mathcal{B} \cap \pi \rvert &= \lvert L_U \rvert= q^{n(r-d)}\left[\frac{dn}{h+1}\right]_q+[r-d]_{q^n}.
        \end{split}
    \end{equation*}
    If $\pi\not= \pi_\infty$, we first determine the number of points at infinity in $\pi\cap \mathcal{B}$ and then determine the number of affine points, that is $|\mathcal B \cap \pi| = |L_T|+\lvert (\mathcal{B} \cap \pi)\setminus L_T |$, where  $L_T:=\pi\cap \pi_\infty\cap L_U$ and whose rank is $w_{L_U}(\pi\cap\pi_{\infty})$. By Corollary \ref{lem:intersectioninfinity}, $L_T$ is a linear set (as $\pi\cap \pi_{\infty}$ defines a hyperplane of $\pi_{\infty}$) such that one of the two cases holds:
\begin{enumerate}
    \item[1)] $L_T$ has rank $n(r-d-1)+\frac{dn}{h+1}$ and $\lvert L_T \rvert=q^{n(r-d-1)}\left[\frac{dn}{h+1}\right]_q+[r-d-1]_{q^n}$;
    \item[2)] $L_T$ has rank $\gamma_i+n(r-d)$ and  
$|L_T|=q^{n(r-d)}[\gamma_i]_q+[r-d]_{q^n}$,
for $i\in \{0,\cdots, h\}$.
\end{enumerate}
Now we find the number of affine points in $\pi \cap \mathcal{B}$, that is $\lvert (\mathcal{B} \cap \pi)\setminus L_T |$.
    From Proposition \ref{prop:weightbyintersectioninfinity} we find that 
    $$\lvert (\mathcal{B} \cap \pi)\setminus L_T |\in \{0, q^{\gamma_i+n(r-d)},q^{\frac{nd}{h+1}+n(r-d-1)}\}$$
    Note that if a hyperplane $\pi'$ through $\pi\cap \pi_{\infty}$ different from $\pi_{\infty}$ has at least one point in $\mathcal{B}\setminus \pi_{\infty}$, then $|\pi'\cap \mathcal{B}|=q^{\rk(L_T)}+|L_T|$. Therefore, any hyperplane of $\PG(r,q^n)$ meets $\mathcal{B}$ in one of the following number of points:
\begin{itemize}
    \item[a)] $q^{n(r-d)}\left[\frac{dn}{h+1}\right]_q+[r-d]_{q^n}$;
    \item[b)] $ q^{n(r-d-1)}\left[\frac{dn}{h+1}\right]_q+[r-d-1]_{q^n}$ (arising from 1));
    \item[c)] $ q^{n(r-d-1)}\left[\frac{dn}{h+1}+1\right]_q+[r-d-1]_{q^n}$ (arising from 1));
    \item[d)] $q^{n(r-d)}[\gamma_i]_q +[r-d]_{q^n}$ (arising from 2));
    \item[e)]  $q^{n(r-d)}([\gamma_i]_q+q^{\gamma_i}) +[r-d]_{q^n}$ (arising from 2)),
\end{itemize}
for any $i \in \{0,\ldots,h\}$.
We prove that hyperplanes meeting $\mathcal{B}$ in the number of points listed in a)-e) always exist expect for the case d) with $i=0$. Indeed, a) is obtained considering the hyperplane at infinity. Consider now $\tau$ a hyperplane of $\pi_{\infty}$, we know that the hyperplanes (different from $\pi_\infty$) through $\tau$ partition the (affine) points of $\mathcal{B}\setminus  \pi_\infty$. This implies the existence of hyperplanes meeting $\mathcal{B}$ in the number points c) and e) (choosing $\tau$ having weight $n(r-d-1)+\frac{dn}{h+1}$ or $\gamma_i+n(r-d)$ with respect to $L_{U'}$, respectively). Moreover, a hyperplane meeting $\mathcal{B}$ in the number of points b) and d) only exists if there exists at least one hyperplane through $\tau$ which does not have any affine point of $\mathcal{B}$.
Using Proposition \ref{prop:weightbyintersectioninfinity}, we know that
$$\lvert\mathcal{B}\setminus \pi_\infty \rvert=q^{\rk(L_U)}.$$
We also know that there are $q^n$ such hyperplanes different from $\pi_\infty$, 
therefore if
$$q^n\cdot q^{w_{L_{U}}(\tau)}=q^{\rk(L_U)},$$
it follows that there do not exist hyperplanes thorugh $\tau$ that intersect $\mathcal{B}\setminus \pi_\infty$ trivially.
Taking into account that $\rk(L_U)= \frac{dn}{h+1}+ n(r-d)$, we find that
there are hyperplanes meeting $\mathcal{B}\setminus \pi_\infty$ trivially except for the cases b), and e) for $i=0$.
This allows us to completely determine the intersection numbers with respect to the hyperplanes, completing the proof.
\end{proof}

\begin{remark}\label{rk:weighthyperconstrB}
In the proof of the above theorem, we may also deduce the possible values of the weight of the hyperplanes with respect to $L_{U'}$, which are:
\begin{itemize}
    \item $n(r-d)+\frac{dn}{h+1}$;
    \item $n(r-d-1)+\frac{dn}{h+1}+1$ (which happens if and only if $r>d$);
    \item $n(r-d)+\gamma_i$;
    \item $n(r-d)+\gamma_j+1$,
\end{itemize}
for any $i,j \in \{0,\ldots,h\}$.
\end{remark}

When choosing $r=d$ in Construction \ref{con:genhyper},  and taking into account Remark \ref{rk:weighthyperconstrB}, Theorem \ref{th:intersectsizes} simplifies to the following.

\begin{corollary}
\label{th:constrnocompl}
Suppose that $\pi_\infty:=\PG(r-1,q^n)$, for $r\geq 2$ and $n\geq 2$, is embedded in $\PG(r,q^n)$. Suppose that $L_{U}$ is a properly maximum $h$-scattered linear set in $\pi_\infty$ and $Q=\langle y\rangle_{\mathbb{F}_{q^n}}$ is a point in $\PG(r,q^n)\setminus\pi_\infty$. If $U'=U \oplus \langle y\rangle_{\mathbb{F}_{q}}$, then the set $\mathcal{B}=L_{U'}$ is a set of type
\begin{center}
$\left( 
\left[\frac{rn}{h+1}\right]_q, [\gamma_i]_q+\beta_{i,j}
\right)_{r-1}$,
\end{center}
where for $i\in\{0,...,h\}$ and $j \in \{0,1\}$, $\gamma_i:=\frac{dn}{h+1}-n+i$, and 
\[\beta_{i,j} = \begin{cases}
0 & \text{if } j=0 \text{ and } i \geq 1\\
q^{\gamma_i} & \text{otherwise}.
\end{cases}
\]
Finally, we have that $\lvert \mathcal{B} \rvert=\left[\frac{rn}{h+1}+1\right]_q$.
\end{corollary}

\begin{remark}
Note that, using Theorem \ref{th:intersectionhyper}, one can actually determine also the number of hyperplanes meeting $\mathcal{B}$ in the number of points described in Corollary \ref{th:constrnocompl} which strongly depends on the $t_i$'s of Theorem \ref{th:intersectionhyper}. 
\end{remark}

\subsection{Second general construction}\label{sec:constr2}

The second construction we analyze corresponds to taking the same affine points as in the previous construction together with its complement at infinity. Also for this construction we are able to determine the pattern of intersection with the hyperplanes.
For the next theorem we refer again to Figure \ref{fig:cons2}.

\begin{theorem}\label{th:intersectsizes2}
Let $\mathcal{D}=L_{U} \subseteq \pi_{\infty}:=\PG(r-1,q^n)$ be as in Construction \ref{con:genhyper} and embed it in $\PG(r,q^n)$. Let $Q=\langle y \rangle_{\F_{q^n}} \in \PG(r,q^n) \setminus \pi_{\infty}$. Let $U'=U \oplus \langle y \rangle_{\F_{q}}$.
Define $\mathcal{K}:=(L_{U'}\setminus L_U)\cup (\pi_\infty \setminus L_U)$. 
 We have that
 $|\mathcal{K}|=q^{n(r-d)}\left([d]_{q^n} -\left[\frac{dn}{h+1}\right]_q+ q^{\frac{dn}{h+1}} \right) $

\begin{center}
\scalebox{0.75}{%
$\left( q^{n(r-d)}\left([d]_{q^n} -\left[\frac{dn}{h+1}\right]_q\right),
q^{n(r-d-1)}\left([d]_{q^n}-\left[\frac{dn}{h+1}\right]_q+ q^{\frac{dn}{h+1}}\right), 
q^{n(r-d)}\left( [d-1]_{q^n}-[\gamma_i]_q+ \beta_{i,j}\right)
\right)_{r-1}$,
}
\end{center}
%
where for $i\in\{0,...,h\}$ and $j \in \{0,1\}$, $\gamma_i:=\frac{dn}{h+1}-n+i$, and 
\[\beta_{i,j} = \begin{cases}
0 & \text{if } j=0 \text{ and } i \geq 1\\
q^{\gamma_i} & \text{otherwise}.
\end{cases}
\]

In particular if $q=2$, we obtain that $\mathcal{K}$ is of type 
\begin{center}
\scalebox{0.75}{%
$\left( 2^{n(r-d)}\left([d]_{2^n} -2^\frac{dn}{h+1}+1\right),
2^{n(r-d-1)}\left([d]_{2^n} +1\right), 
2^{n(r-d)}\left([d-1]_{2^n} -2^{\gamma_i}+1+ \beta_{i,j}\right)
\right)_{r-1}$
}
\end{center}
which in total gives $2h+3$ possibilities. In particular, $|\mathcal{K}|=2^{n(r-d)}\left( [d]_{2^n}+1 \right)$.
\end{theorem}
\begin{proof}
The intersection sizes with a hyperplane $\pi$ and $\mathcal{K}$ can be derived from Lemma \ref{lem:intersectioninfinity} and Proposition \ref{prop:weightbyintersectioninfinity} for the affine part. In particular, we always have that
\begin{equation}\label{eq:countK}
    \begin{split}
        \lvert \mathcal{K} \cap \pi \rvert &= \lvert \mathcal{K} \cap \pi \cap \pi_\infty \rvert +\lvert (\mathcal{K} \cap \pi)\setminus \pi_\infty  \rvert\\
        &=
        \begin{cases}
        [r]_{q^n}-|L_U \cap \pi_{\infty}| & \mbox{ if } \pi=\pi_{\infty},\\
        [r-1]_{q^n}-|L_U \cap \pi|+\lvert (\mathcal{K} \cap \pi)\setminus \pi_\infty \rvert & \mbox{ if } \pi \neq \pi_{\infty}.
        \end{cases}
    \end{split}
\end{equation}

Here we can find $\lvert L_U \cap \pi\rvert$ for every type of hyperplane, this is one of the values determined in Corollary \ref{lem:intersectioninfinity}.
Using Proposition \ref{prop:weightbyintersectioninfinity}, we find that $\lvert (\mathcal{K} \cap \pi)\setminus \pi_\infty  \rvert\in \{0,q^{\rk(L_T)}\}$, where $L_T:=L_U \cap \pi_{\infty}\cap \pi$. 
Noting that the number of hyperplanes $\pi\ne\pi_{\infty}$ meeting $\mathcal{K}$ in at least one point is $[r-1]_{q^n}-|L_U \cap \pi|+\lvert (\mathcal{K} \cap \pi)\setminus \pi_\infty \rvert$ corresponds to the set of hyperplanes meeting $L_{U'}$ in $|L_U \cap \pi|+\lvert (\mathcal{K} \cap \pi)\setminus \pi_\infty \rvert$, which we already counted in the proof of Theorem \ref{th:intersectsizes}. Putting together the arguments of the proof of Theorem \ref{th:intersectsizes} and \eqref{eq:countK} we obtain the assertion.


\end{proof}

We end this section with the following corollary that gives a rather interesting case of the previous theorem ($r=d$), leading to an interesting geometrical object that will be classified in the next section as an hypercylinder.

\begin{corollary}
Suppose that $\pi_\infty:=\PG(r-1,q^n)$, for $r\geq 2$ and $n\geq 2$, is embedded in $\PG(r,q^n)$. Suppose that $L_{U}$ is a properly maximum $h$-scattered linear set in $\pi_\infty$ and $Q=\langle y\rangle_{\mathbb{F}_{q^n}}$ is a point in $\PG(r,q^n)\setminus\pi_\infty$. Let $U'=U \oplus \langle y\rangle_{\mathbb{F}_{q}}$ and define $\mathcal{K}:=(L_{U'}\setminus L_U)\cup (\pi_\infty \setminus L_U)$, then $\mathcal{K}$ is of type

\begin{center}
\scalebox{1}{%
$\left(  [r]_{q^n}- \left[\frac{rn}{h+1}\right]_q,
[r-1]_{q^n}-[\gamma_i]_q+ \beta_i
\right)_{r-1}$
}
\end{center}
Here  $\gamma_i:=\frac{rn}{h+1}-n+i$ for $i\in\{0,...,h\}$, and $\beta_i\in \{q^{\gamma_i},0\}$ for $i\geq 1$ or $\beta_0=q^{\gamma_0}$. Finally, we also conclude that
$|\mathcal{K}|= [r]_{q^n}-\psi\left[\frac{rn}{h+1}\right]_q+ q^{\frac{rn}{h+1}}.$\\

In particular if $q=2$ and $h=1$, we obtain that $\mathcal{K}$ is of type 
\begin{center}
\scalebox{1}{%
$\left(  [r]_{2^n}-2^\frac{rn}{h+1}+1,
[r-1]_{2^n}+1,
[r-1]_{2^n}-2^{\gamma_1}+1
\right)_{r-1}$
}
\end{center}
which in total gives $3$ possibilities. Note that $|\mathcal{K}|= [r]_{2^n}+1 $.
\end{corollary}
\begin{proof}
This follows directly combining Theorem \ref{th:intersectsizes2} and Corollary \ref{lem:intersectioninfinity}.
\end{proof}

\section{Stability results for hypercylinders}\label{sec:stab}

When considering $q=2,d=2$ and $h=1$, in Theorem \ref{th:intersectsizes2}, we obtain examples of point sets of $\PG(r,2^n)$ with few intersection numbers with respect to the hyperplanes, more precisely we have constructed point sets $S \subseteq \PG(r,2^n)$ of size $2^{n(r-1)}+2^{n(r-2)+1}$ of type $(0,2^{n(r-2)}+2^{n(r-3)+1},2^{n(r-2)+1})_{r-1}$.
In this section we will study point sets with size $q^{r-1}+q^{r-2}+t$ in $\PG(r,q)$ of type $(0,q^{r-2}+2q^{r-3},q^{r-2}+t)_{r-1}$, under certain restrictions on $t$. Actually, it turns out that the only objects with these properties are exactly the hypercylinders, that are hyperoval cones with their vertex deleted.

We will divide the discussion in two cases: we will first analyze the case $r=3$ and then, using induction arguments, we will deal with the $r \geq 4$ case. 
\subsection{\texorpdfstring{The case $r=3$.}{The case r=3.}}
We start by giving some geometrical properties of such point sets.

\begin{theorem} \label{th:propgenkmarc} 
Let $S \subseteq \PG(3,q)$ be a set of points of type $\{0,q+2,q+t\}_2$ of size $q^2+q+t$, where $2<t \leq q+1$. Then:
\begin{itemize}
    \item[(i)]  there are no tangent lines to $S$;
    \item[(ii)] through every point of $S$ there are exactly $q^2+q$ lines that are $2$-secant and one $t$-secant line;
    \item [(iii)] all the planes through a $t$-secant line are $(q+t)$-secant planes;
    \item[(iv)] through every $2$-secant line there is exactly one $(q+t)$-secant plane and all the others are $(q+2)$-secant planes; 
    \item[(v)] the intersection of a $(q+2)$-secant plane and $S$ is a hyperoval;
    \item[(vi)] it holds that $q=2^n$, for some positive integer $n$;
    \item[(vii)] $S$ is of type $(0,2,t)_1$;
    \item[(viii)] if $\pi$ is a $(q+t)$-secant plane, then $S\cap \pi$ is a KM-arc of type $t$ in $\pi$;
    \item[(ix)] it holds that $t=2^i$, with $i \leq n$.
\end{itemize}
\end{theorem}
\begin{proof}
(i) By contradiction, suppose that there exists a tangent line $\ell$ to $S$ for a point $P \in S$. All the planes through $\ell$ meet $S$ in at least $q+2$ points. By counting the number of points of the intersection between these planes and $S$, we obtain that
\[
(q+1)(q+1)+1 \leq \lvert S \rvert =q^2+q+t,
\]
from which $t\geq q+2$, a contradiction. \\
(ii) \& (iii) Let $P \in S$. Suppose that no $2$-secant lines pass through $P$. By point (i), we get that $\lvert \ell \cap S \rvert \geq 3$, for every line $\ell$ through $P$. Then, since there are $q^2+q+1$ lines through $P$, we now get
\[
2(q^2+q+1)+1 \leq \lvert S \rvert=q^2+q+t,
\]
and hence $t \geq q^2+q+3$, a contradiction. This implies the existence of at least one $2$-secant line through $P$.\\
Now, suppose that every line through $P$ is a $2$-secant line. Then
\[
q^2+q+1+1= \lvert S \rvert =q^2+q+t,
\]
that is $t=2$, a contradiction. So there exists at least one line $\ell$ through $P$ which is not a $2$-secant line, hence $\lvert \ell \cap S \rvert \geq 3$. 
Now, we prove that all the planes through $\ell$ are $t$-secant planes. 
Suppose that $\ell$ is contained in a $(q+2)$-secant plane $\pi$.
Since $ \lvert S \cap (\pi \setminus \ell) \rvert \leq q-1$, through $P$ there is at least one tangent line contained in $\pi$, which contradicts (i). 
Therefore, every plane through $\ell$ is a $(q+t)$-secant plane to $S$.
Moreover, let $x = \lvert \ell \cap S \rvert$. Since the planes through $\ell$ give a partition of $S \setminus \ell$ and since each plane contains $q+t-x$ elements of $S$ that are not in $\ell$, we get that
\[
(q+1)(q+t-x)+x=q^2+q+t,
\]
hence $x=t$ and $\ell$ is a $t$-secant line. 
So, we have proved that every line through $P$ is either a $2$-secant line or a $t$-secant line.
In particular, the number of $t$-secant lines through $P$ is one. 
Indeed, suppose that there are at least two $t$-secant lines through $P$. Then 
\[
2(t-1)+q^2+q-1+1 \leq q^2+q+t,
\]
and so $t \leq 2$, a contradiction. \\
(iv) Let $\ell$ be a $2$-secant line.
Let $x$ denote the number of $(q+t)$-secant planes through $\ell$.
Then
\[
 q^2 + q + t = |S| = 2 + x (q+t-2) + (q+1-x) q,
\]
hence $x=1$. \\
(v) Let $\pi$ be a $(q+2)$-secant plane. By (ii), the lines meet $S$ either in the empty set, in $2$ points or in $t$ points.
Moreover, if there exists a $t$-secant line in $\pi$ through a point of $S \cap \pi$ then by (iii) $\pi$ is a $(q+t)$-secant plane, a contradiction. So, in $\pi$ there are only external lines and $2$-secant lines and hence $S \cap \pi$ is a hyperoval. \\
(vi) As a consequence of the previous point $q$ is even.\\
(vii) In a $(q+2)$-secant plane $\pi$ there exist external lines to $S$. Together with (ii), this implies that $S$ is of type $(0,2,t)_1$ in $\pi$. \\
(viii) Let $\ell$ be a $t$-secant line and let $\pi$ be a plane through $\ell$. By (iii), $\pi$ is a $(q+t)$-secant plane.
This means that $S'=S\cap \pi$ has size $q+t$ and by (vii) it is of type $(0,2,t)_1$ in $\pi$. Therefore, $S'$ is a KM-arc of type $t$ in $\pi$. \\
(ix) Let $\pi$ be a $(q+t)$-secant plane. By (viii), $S\cap \pi$ is a KM-arc of type $t$ in $\pi$. Then, by Proposition \ref{prop:kmarcs}, $t \mid q$ and the assertion then follows from the fact that $q=2^n$.
\end{proof}


We are now able to completely characterize such point sets in $\PG(3,q)$.

\begin{corollary}\label{cor:charr=3}
Suppose that $q \geq 4$ and let $S \subseteq \PG(3,q)$ be a set of points of type $\{0,q+2,q+t\}_2$ of size $q^2+q+t$, with $2<t \leq q+1$. Then $q$ is even, $t=q$ and $S$ is a hypercylinder.
\end{corollary}
\begin{proof}
By Theorem \ref{th:propgenkmarc}, we find that $q$ is even, say $q=2^n$ for some positive integer $n \geq 2$, $t=2^i$ and $S$ is of type $(0,2,2^i)_1$. This means that the intersection numbers of $S$ with respect to the lines are even and hence, by Theorem \ref{Th:minimumsizecodeword}, it follows that $q^2+q+t\geq q^2+2q$. In addition, since $t=2^i$, we have $i=n$.
Therefore, since $S$ is a point set of size $q^2+2q$ and it is of type $(0,2,q)_1$, from Theorem \ref{th:classconehyperoval} the assertion follows.
\end{proof}

\subsection{\texorpdfstring{The case $r\geq 4$.}{The case r>=4.}}
We will now deal with the point sets in $\PG(r,q)$ with $r \geq 4$. We start by deriving some geometrical properties. 

\begin{theorem} \label{th:propgenkmarcspace}
Let $S \subseteq \PG(r,q)$, with $r\geq 4$ and $q>2$, be a set of points of type $\{0,q^{r-2}+2q^{r-3},q^{r-2}+t\}_{r-1}$  of size $q^{r-1}+q^{r-2}+t$, with $2q^{r-3}< t \leq q^{r-2}+q-1$. Then the following properties are true:
\begin{itemize}
    \item[(i)] if $\pi$ is a $k$-space with $2 \leq k < r$, then $\pi \cap S = \emptyset$, or $|\pi \cap S| \geq q^{k-1}+2q^{k-2}$;
    \item[(ii)] there are no tangent lines to $S$;
    \item[(iii)] there exists at least one $k$-space that is $(q^{k-1}+2q^{k-2})$-secant, for every $k \in \{2,\ldots,r-1\}$, and at least one $2$-secant line; 
    \item[(iv)] any $(q+2)$-secant plane to $S$ meets $S$ in a hyperoval;
    \item[(v)] it holds that $q$ is even;
    \item[(vi)] all hyperplanes through an $(q^{r-3}+2q^{r-4})$-secant $(r-2)$-space are $(q^{r-2}+2q^{r-3})$-secant spaces; 
    \item[(vii)] it holds that $t=q^{r-2}$;
    \item[(viii)] every $(r-i+1)$-space through a fixed $(q^{r-i-1}+2q^{r-i-2})$-secant $(r-i)$-space is a $(q^{r-i}+2q^{r-i-1})$-secant space, for any $i \in \{1,\ldots,r-2\}$. 
\end{itemize}
\end{theorem}
\begin{proof}
(i) 
We prove this property by induction on the $(r-i)$-spaces.
The case $i=1$ follows from our assumptions.
Now suppose that $i \in \{1,\ldots,r-3\}$, and that all the secant $(r-i)$-spaces meet $S$ in at least $q^{r-i-1}+2q^{r-i-2}$ points.
Our aim is to prove that all the secant $(r-i-1)$-spaces meet $S$ in at least $q^{r-i-2}+2q^{r-i-3}$ points. By contradiction, suppose that there exists a $y$-secant $(r-i-1)$-space $\pi$, with $0 < y < q^{r-i-2}+2q^{r-i-3}$.
Then the $(r-i)$-spaces through $\pi$ meet $S$ in at least $q^{r-i-1}+2q^{r-i-2}$ points (by induction). These intersections give a partition of the points of $S$ not in $\pi$. This implies, by using Proposition \ref{prop:numbersubapcescontaining}, that
\[
[i+1]_q(q^{r-i-1}+2q^{r-i-2} -y)+y \leq q^{r-1}+q^{r-2}+t  
\]
and since $y \leq q^{r-i-2}+2q^{r-i-3}-1$, we get
\[2q^{r-2}+q[i]_q \leq q^{r-2}+t,\] which contradicts $t \leq q^{r-2}+q-1$. \\
(ii) We prove now that all the secant lines meet $S$ in at least two points. Again, suppose by contradiction that there exists a tangent line $\ell$ to $S$. The planes through $\ell$ intersect $S$ in at least $q+2$ points and they form a partition of the points of $S$ not in $\ell$. Similarly as before, this implies, by Proposition \ref{prop:numbersubapcescontaining}, that
\[
[r-1]_q(q+1)+1 \leq q^{r-1}+q^{r-2}+t
\]
thus $2[r-1]_q-q^{r-2} \leq t$, which is a contradiction to the fact that $t\leq q^{r-2}+q-1$ and $r \geq 4$.\\ 
(iii) Let us start by proving that through every point of $S$ there is at least one $2$-secant line.
Let $P \in S$ and suppose that there does not exist any $2$-secant line through $P$. By (ii), we have that the lines through $P$ meet $S$ in at least $3$ points. So,
\[
2[r]_q+1\leq q^{r-1}+q^{r-2}+t
\]
and hence, we find that $\frac{q^r+q^{r-1}+q-3}{q-1}-q^{r-2} \leq t$ a contradiction. 
Let $\ell$ be a $2$-secant line. We now claim that there is at least one $(q+2)$-secant plane through $\ell$. Indeed, if all the planes through $\ell$ intersect $S$ in at least $q+3$ points then 
\[
(q+1)[r-1]_q +2\leq q^{r-1}+q^{r-2}+t
\]
and thus $2[r-1]_q+1 -q^{r-2}\leq t$, a contradiction.

Now, we proceed by induction. 
Take $i \in \{2,\ldots,r-3\}$, and suppose that there exists a $(q^{r-i-2}+2q^{r-i-3})$-secant $(r-i-1)$-space $\rho$.
If no $(r-i)$-space through $\rho$ meets $S$ in $q^{r-i-1}+2q^{r-i-2}$ points, by (i) they all meet $S$ in more than $q^{r-i-1}+2q^{r-i-2}$ points, thus
\[
(q^{r-i-1}+q^{r-i-2}-2q^{r-i-3}+1)[i+1]_q+q^{r-i-2}+2q^{r-i-3} \leq q^{r-1}+q^{r-2}+t
\]
and hence $q^{r-2}+[i+1]_q \leq t$, which contradicts our hypothesis on $t$. 
So, there exists some $(q^{r-i-1}+2q^{r-i-2})$-secant $(r-i)$-space through $\rho$.\\
(iv) \& (v) Let $\pi$ be $(q+2)$-secant plane to $S$.
Since there are no tangent lines to $S$, $S\cap \pi$ is a hyperoval and $q$ has to be even.\\
(vi) Let $\pi$ be an $(r-2)$-space that is $(q^{r-3}+2q^{r-4})$-secant. By hypothesis, all hyperplanes through $\pi$ meet $S$ in either $q^{r-2}+2q^{r-3}$ or $q^{r-2}+t$ points. Suppose that there exists at least one hyperplane through $\pi$ that is $(q^{r-2}+t)$-secant. 
This implies that
\[
q(q^{r-2}+q^{r-3}-2q^{r-4})+q^{r-2}+t-(q^{r-3}+2q^{r-4})+(q^{r-3}+2q^{r-4}) \leq q^{r-1}+q^{r-2}+t
\]
and so $q^{r-2}-2q^{r-3} \leq 0$.
This is only possible if $q=2$, a contradiction. \\
(vii) Let $\pi$ be a $(q^{r-3}+2q^{r-4})$-secant $(r-2)$-space, which always exists because of (iii).
By (vi), we know that all the hyperplanes through $\pi$ are $(q^{r-2}+2q^{r-3})$-secant.
Note that they provide a partition of the points of $S$ outside $\pi$.
This implies that
\[
(q+1)(q^{r-2}+q^{r-3}-2q^{r-4})+q^{r-3}+2q^{r-4}= q^{r-1}+q^{r-2}+t,
\]
and so $t=q^{r-2}$. \\
(viii) Let $\pi$ be an $(r-i)$-space which is $(q^{r-i-1}+2q^{r-i-2})$-secant.
Then the number of $(r-i+1)$-spaces through $\pi$ equals $[i]_q$.
Since each point of $S \setminus \pi$ is contained in a unique such $(r-i+1)$-space and by (i) all $(r-i+1)$-spaces through $\pi$ meet $S$ in at least $q^{r-i}+2q^{r-i-1}$ points, we find that
\[
q^{r-1}+2q^{r-2}=|S| \geq (q^{r-i}+q^{r-i-1}-2q^{r-i-2})[i]_q+q^{r-i-1}+2q^{r-i-2} = q^{r-1}+2q^{r-2},
\]
so the above is an equality, from which it follows that all of these $(r-i)$-spaces are $(q^{r-i}+2q^{r-i-1})$-secant.
\end{proof}

Thanks to the geometry of the point sets described in the previous result, we are able to determine their possible intersection numbers with respect to the $k$-spaces.

\begin{theorem} 
\label{th:intersectionspaces}
Suppose that $q\geq 4$ and let $S \subseteq \PG(r,q)$ be a set of points of type $\{0,q^{r-2}+2q^{r-3},2q^{r-2}\}_{r-1}$ of size $q^{r-1}+2q^{r-2}$.
If $\tau$ is an $(r-i)$-space, with $1 \leq i \leq r-2$, then
\[
 |\tau \cap S| \in \{2q^{r-i-1}+c(-q^{r-i-1}+2q^{r-i-2}) \colon c \in \mathbb{Z}, -q \leq c \leq 1\} \cup \set 0.
\]
\end{theorem}
\begin{proof}
Let's start by proving the result when $i=2$. Let $\tau$ be a $j$-secant $(r-2)$-space. Let $m$ be the number of hyperplanes through $\tau$ of $\PG(r,q)$ that are $(q^{r-2}+2q^{r-3})$-secant. So, by counting arguments, we find that
\[
m(q^{r-2}+2q^{r-3}-j)+(q+1-m)(2q^{r-2}-j)+j=q^{r-1}+2q^{r-2},
\]
from which it follows that
\[
m(-q^{r-3}+2q^{r-4})=-q^{r-2}+j.
\]
This means that $-q^{r-3}+2q^{r-4}$ divides $-q^{r-2}+j$, thus $-q^{r-3}+2q^{r-4}$ divides $-2q^{r-3}+j$. Therefore, there exists an integer $c$ such that $j=2q^{r-3}+c(-q^{r-3}+2q^{r-4})$. By (i) of Theorem \ref{th:propgenkmarcspace} we also have 
that $c \leq 1$.\\
Now we use an inductive argument.
Suppose that the assertion is true for all the $(r-g)$-spaces with $g\leq i$, for some $i \in \{2,\ldots,r-3\}$. 
Then we want to prove the assertion for $g=i+1$.
In particular, for every secant $(r-i+1)$-spaces, its intersection number with $S$ is in the set
 $\{2q^{r-i}+c(-q^{r-i}+2q^{r-i-1}) \colon c \in \mathbb{Z}, c \leq 1\}$, and all the secant $(r-i)$-space have possible intersection numbers in the set $\{2q^{r-i-1}+c(-q^{r-i-1}+2q^{r-i-2}) \colon c \in \mathbb{Z}, c \leq 1\} $.
We claim that 
\[
 |\tau \cap S| \in \{2q^{r-i-2}+c(-q^{r-i-2}+2q^{r-i-3}) \colon c \in \mathbb{Z}, c \leq 1\}, 
\]
for every secant $(r-i-1)$-space $\tau$. First, assume that $\tau$ is a $j$-secant $(r-i-1)$-space. This means that there exists some $\overline{c} \in \mathbb{Z}$ such that $\tau$ is contained in a $(2q^{r-i}+\overline{c}(-q^{r-i}+2q^{r-i-1}))$-secant $(r-i+1)$-space $\sigma$. 
Denote the number of $(r-i)$-spaces through $\tau$ that are $(2q^{r-i-1}+\alpha(-q^{r-i-1}+2q^{r-i-2}))$-secant and contained in $\sigma$ by $m_\alpha$. As before, by counting the size of the intersection between $(r-i)$-spaces through $\tau$ and $S$ we obtain
 \[
 \sum_{\alpha} m_{\alpha} (2q^{r-i-1}+{\alpha}(-q^{r-i-1}+2q^{r-i-2})-j)+j=2q^{r-i}+\overline{c}(-q^{r-i}+2q^{r-i-1}), 
 \]
and since $\sum_{\alpha} m_{\alpha}=q+1$, we have that
\[
\begin{array}{rl}
qj & = 2q^{r-i-1}+\sum_{\alpha} m_{\alpha} {\alpha}(-q^{r-i-1}+2q^{r-i-2})-\overline{c}(-q^{r-i}+2q^{r-i-1}) \\
& = 2q^{r-i-1}+(\sum_{\alpha} m_{\alpha} {\alpha}-\overline{c}q)(-q^{r-i-1}+2q^{r-i-2})
\end{array}\] 
and so
\[
j=2q^{r-i-2}+a(-q^{r-i-2}+2q^{r-i-3}),
\]
for some integer $-q\leq a \leq 1$. 
\end{proof}

As a consequence of Theorem \ref{th:propgenkmarcspace}, we can focus solely on the case in which $q$ is even. We now are able to determine all the possible intersection sizes with the lines (note that the lines were not covered in the above theorem).

\begin{corollary} \label{cor:sisevenset}
Let $S \subseteq \PG(r,q)$, with $q$ even and $q\geq 4$, be a set of points of type $\{0,q^{r-2}+2q^{r-3},2q^{r-2}\}_{r-1}$ of size $q^{r-1}+2q^{r-2}$. Then $S$ is of type $(0,2,q)_1$. 
\end{corollary}

\begin{proof}
Let $\ell$ be a $j$-secant line to $S$ with $2\leq j \leq q+1$. By Theorem  \ref{th:intersectionspaces}, we have that 
\[ |\pi \cap S| \in \{2q+c(-q+2) \colon c \in \mathbb{Z},-q \leq c \leq 1\}, \]
for every plane $\pi$ and 
\[
 |\tau \cap S| \in \{2q^{2}+c(-q^{2}+2q) \colon c \in \mathbb{Z},-q \leq c \leq 1\},
\]
for every $3$-space $\tau$. This implies that there exists some $\overline{c}$ such that $\ell$ is contained in a $(2q^{2}+\overline{c}(-q^{2}+2q))$-secant $3$-space $\overline{\tau}$.
Denote by $m_{\alpha}$ the number of planes through $\ell$ that are $(2q+{\alpha}(-q+2))$-secant and contained in $\overline{\tau}$. So we get
\[
 \sum_{\alpha} m_{\alpha} (2q+{\alpha}(-q+2)-j)+j=2q^{2}+\overline{c}(-q^{2}+2q),
\]
and since $\sum_{\alpha} m_{\alpha}=q+1$, we have that
\[
2q-jq=\sum_{\alpha} m_{\alpha} {\alpha}(q-2)+\overline{c}(-q^{2}+2q)
\]
and 
\[
2q-jq=\left(\sum_{\alpha} m_{\alpha} {\alpha}-q\overline{c}\right)(q-2).
\]
Therefore $q-2$ divides $jq-2q$. Since $q$ is a power of $2$ and $\frac{q-2}2$ is odd, we find that $q-2$ divides $2j-4$. In this way $j=2+a(\frac{q-2}{2})$, for some integer $a$. Since $2 \leq j \leq q+1$, we get that $j \in \{2,q/2+1,q\}$. Now suppose that there exists a $(q/2+1)$-secant line $s$. By Theorem \ref{th:intersectionspaces}, any plane meets $S$ in $q+2+\alpha(q-2)$ points, for some $\alpha$ with $1 \leq \alpha \leq 1+q$.
Let $g_\alpha$ be the number of $q+2+\alpha(q-2)$-secant planes through $s$, for any $\alpha \in \{0,\ldots,q+1\}$. Clearly $\sum_{\alpha=0}^{q+2} g_{\alpha}=[r-1]_q$ from which we obtain that
\[
\sum_{\alpha=0}^{q+1}g_{\alpha}(q+2+\alpha(q-2)-\frac{q+2}{2} )+\frac{q+2}{2}=q^{r-1}+2q^{r-2}.
\]
This implies that 
\[
\sum_{\alpha=0}^{q+1}g_\alpha \alpha=\frac{(q+2)(q^{r-2}-1)}{2(q-1)} \geq \sum_{\alpha=1}^{q+1}g_\alpha =[r-1]_q-g_0
\]
and hence $g_0 \geq 1$. It follows that through $s$ there is at least a $(q+2)$-secant plane, this is a contradiction to (iv) of Theorem \ref{th:propgenkmarcspace}. This shows that $S$ is of type $\{0,2,q\}_1$. Now, it is easy to see that not all of the lines through a point of $S$ can be $2$-secant lines and not all of them can be $q$-secant lines, due to the size of $S$. Therefore, $S$ is of type $(0,2,q)_1$.
\end{proof}

So, we have reduced the problem of studying our point sets to study minimum size even sets with some other geometrical assumptions. These latter sets can be characterized as follows.

\begin{theorem} \label{th:classificationsam}
Let $S \subseteq \PG(r,q)$ be an even set of size $q^{r-1}+2q^{r-2}$. Assume that there exists at least one $(q+2)$-secant plane and that every solid through a $(q+2)$-secant plane is a hypercylinder. Then $S$ is a hypercylinder.  
\end{theorem}
\begin{proof}
Let $\pi$ be a $(q+2)$-secant plane. By hypothesis, we have that $S \cap \Sigma$ is a hypercilinder for each solid $\Sigma$ through $\pi$.
Every such hypercilinder has a point $V$ as vertex.
Let $\tau$ denote the set consisting of these vertices.Then $|\tau|$ corresponds to the number of $3$-spaces through a plane, that is $\theta_{r-3}$.
We will show that $\tau$ is a subspace, necessarily of dimension $r-3$.
Then $S$ must be a hypercilinder with vertex $\tau$ and as base $\pi \cap S$.
To show that $\tau$ is a subspace, we will prove that if a line contains more than 1 point of $\tau$, it is completely contained in $\tau$.
To this aim take two distinct vertices $V_1$ and $V_2$ in $\tau$.
Denote the line through $V_1$ and $V_2$ by $v$.
Since $V_1$ and $V_2$ lie in distinct solids through $\pi$, $v$ cannot intersect $\pi$.
Take a point $Q \in \pi \cap S$ and consider the plane $\sigma = \langle v, Q\rangle $.
Then for $i=1,2$, it must be that $\langle V_i,Q \rangle \cap S = \langle V_i,Q \rangle \setminus \{V_i\}$.
Take a line $l$ in $\sigma$ through $Q$, with $l \neq \langle V_i,Q\rangle $ for $i=1,2$.
Consider the solid $\Sigma = \vspan{\pi,l}$.
Suppose that the vertex $V_3$ of this solid does not lie on $l$.
Then there exists some plane $\rho$ in $\Sigma$ through $l$, that doesn't contain $V_3$.
Then $\rho$ intersects $\pi$ in a line.
Note that $\rho \cap S$ cannot be empty, since it contains $Q$, and cannot be the symmetric difference of two lines, since it doesn't contain $V_3$.
\\
Note that for every plane $\rho$ that intersect $\pi$ in a line, $\rho \cap S$ is either empty, a hyperoval, or the symmetric difference of two lines.
This follows immediately from the fact that $\pi$ and $\rho$ span a solid, this solid must intersect $S$ in a hypercilinder, and a hypercilinder in $\PG(3,q)$ intersects every plane in the empty set, a hyperoval or the symmetric difference of two lines.
\\
This means that $S \cap \rho$ is a hyperoval and hence $l$ is a $2$-secant line.
Since $\sigma$ intersects $\rho$ in a line, this means that $\sigma \cap S$ is either empty or a hyperoval or the symmetric difference of two lines.
All of these options are impossible, given how $S$ intersects the lines $\langle {V_i,Q}\rangle $.
Thus, $l$ must contain $V_3$.
As a consequence, every line $l$ through $Q$ in $\sigma$ intersects $S$ in $q$ points, and the point $l \setminus S$ is a vertex of a solid through $\pi$.
Then the points of $\sigma \setminus S$ are set of $q+1$ points, intersecting every line of $\sigma$ in an odd number of points since $S$ is an even set.
In particular, they are a blocking set of size $q+1$, hence a line, necessarily $v$.
This proves that every point of $v$ is a vertex, hence that $\tau$ is a subspace.
\end{proof}

We can finally conclude our characterization result.

\begin{corollary}\label{cor:hypercylinder}
Suppose that $r\geq 3$ and $q\geq 4$. Let $S \subseteq \PG(r,q)$ be a set of points of type $\{0,q^{r-2}+2q^{r-3},q^{r-2}+t\}_{r-1}$ of size $q^{r-1}+q^{r-2}+t$, with $2q^{r-3}< t \leq q^{r-2}+q-1$. Then $q$ is even, $t=q^{r-2}$ and $S$ is a hypercylinder.
\end{corollary}
\begin{proof}
The case $r=3$ is ruled out by Corollary \ref{cor:charr=3}.
By (v), (vii) and (viii) of Theorem \ref{th:propgenkmarcspace}, it follows that $q$ is even, $S$ has size $q^{r-1}+2q^{r-2}$ and there is at least one $(q+2)$-secant plane to $S$. Moreover, by Corollary \ref{cor:sisevenset}, we know that $S$ is an even set and by (iv) of Theorem \ref{th:propgenkmarcspace}, we have that all the the solids $\Sigma$ through a $(q+2)$-secant plane meet $S$ in $q^2+2q$ points and hence by Theorem \ref{th:classconehyperoval} we obtain that $\Sigma\cap S$ is a hypercylinder. The assertion now follows by Theorem \ref{th:classificationsam}.
\end{proof}

\section{Some consequences on linear codes}\label{sec:codes}

\subsection{Linear codes}

A \textbf{linear code} $\C$ is any $\F_q$-subspace of $\F_{q}^n$.
If $\C$ has dimension $k$, we say that $\C$ is an $[n, k]_q$-\textbf{code}.
A \textbf{generator matrix} $G$ for $\C$ is a $(k \times n)$-matrix over $\F_q$ whose rows form a basis of $\C$, i.e.
\[ \C=\{xG \colon x\in \F_q^k\}. \]

The \textbf{Hamming distance} in $\F_q^n$ is defined as follows: $d(x,y)$ is the number of entries in which $x$ and $y$ differ, with $x,y \in \F_q^n$.
So we can define the \textbf{minimum distance} of a code $\C$ as follows 
\[d=d(\C)=\min\{d(x,y) \colon x,y \in \C, x \ne y\}.\]
A linear code in $\fq^n$ of dimension $k$ and minimum distance $d$ is said to be an $[n, k, d]_q$-code.
The \textbf{weight} $w(x)$ of a vector $x \in \fq^n$ is the number of positions of the non-zero components of $x$. 
When the code is linear the minimum weight of a code coincides with its minimum distance. 
A code $\C$ is said to be \textbf{nondegenerate} if there does not exist a position in which all the codewords of $\C$ have zero as entry.
Let $A_i$ be the number of codewords of $\C$ with Hamming weight $i$. 
Then the \textbf{weight distribution} of $\C$ is $(A_0,\ldots,A_n)$.

Two $[n,k,d]_q$-codes $\C_1$ and $\C_2$ are said to be (monomially) \textbf{equivalent} if 
\[\C_2=\C_1 PD=\{cPD \colon c\in \C_1\},\] 
for some permutation matrix $P \in \F_q^{n \times n}$ and invertible diagonal matrix $D \in \F_q^{n\times n}$. 

The set of equivalence classes of nondegenerate $[n,k,d]_q$--codes will be denoted by $\classcode$.

Furthermore, $ \mathcal C$ is called \textbf{projective} if in one (and thus in all) generator matrix $G$ of $\mathcal C$ no two columns are proportional. Note that a projective code is necessarily nondegenerate.

Linear codes have a natural geometric representation via projective systems. A \textbf{projective $[n, k, d]_q$--system} $(\mathcal{P},\mathrm{m})$ is a multiset, where $\mathcal{P} \subseteq \mathrm{PG}(k - 1, q)$ is a set of points  not
all of which lie in a hyperplane, and $\mathrm{m}: \PG(k-1,q) \rightarrow \mathbb{N}$ is the multiplicity function, with $\mathrm{m}(P)>0$ if and only if $P \in \mathcal{P}$ and $\sum_{P \in \mathcal{P}}\mathrm{m}(P)=n$. The parameter $d$ is defined by
\[n - d = \max\left\{\sum_{P \in H} \mathrm{m}(P): H \leq \PG(k-1,q), \dim(H)=k-2\right\},\]
see  e.g.\ \cite{vladut2007algebraic}.

Similar to linear codes, we will say that two projective $[n,k,d]_q$--systems $(\mathcal{P},\mathrm{m})$ and $(\mathcal{P}',\mathrm{m}')$  are \textbf{equivalent} if there exists $\phi \in \mathrm{PGL}(k,q)$ mapping $\mathcal{P}$ to $\mathcal{P}'$ that preserves the multiplicities of the points, i.e.
$\mathrm{m}(P)=\mathrm{m}'(\phi(P))$ for every $P \in \PG(k-1,q)$. The set of all equivalence classes of projective $[n,k,d]_q$ systems will be denoted by $\mathcal{P}[n,k,d]_q$.

\subsection{Linear codes and projective systems}\label{sec:projsystlincod}

In this subsection we briefly describe the well-known connection between projective systems and nondegenerate codes.
Indeed, there exists a one-to-one correspondence between $\classcode$ and $\classsystem$. 

For a given nondegenerate $[n,k,d]_q$--code $\C$, consider a generator matrix $G \in \F_{q}^{k \times n}$. Let $g_i$ be the columns of $G$ and define the set $\mathcal{P}=\{\langle g_1\rangle_{\fq},\ldots,\langle g_n\rangle_{\fq}\} \subseteq \PG(k-1,q)$. Moreover, define the multiplicity function $\mathrm{m}$ as
\[
\mathrm{m}(P)=\lvert \{i: P=\langle g_i\rangle_{\fq}\} \rvert.
\]
Then, a projective system associated with $\C$ is $(\mathcal{P},\mathrm{m})$. 
On the other hand, given a projective $[n,k,d]_q$--system $(\mathcal{P},\mathrm{m})$, we can construct a matrix $G$ by taking as columns representatives of each point $P_i$ in $\mathcal{P}$, counted with multiplicity $\mathrm{m}(P_i)$. Let $\C$ be the space generated by the rows of $G$. We say that $\C$ is a code associated with $(\mathcal{P},\mathrm{m})$. 
This naturally yields the one-to-one correspondence between $\classcode$ and $\classsystem$.

Also, the weight distribution of the code associated with a projective system is related to its intersection pattern with hyperplanes. More precisely, let $\C$ be a nondegenerate $[n,k,d]_q$-code. Let $G \in \F_q^{k \times n}$ be a generator matrix of $\C$. Let $g_i$, for $i\in \{1,\ldots,n\}$ be the $i$-th column of $G$. The weight of a codeword $vG \in \C$ is
\[
w(vG)=n-\lvert \{i:v \cdot g_i =0\} \rvert .
\]
Equivalently, consider $(\mathcal{P},\mathrm{m})$ the multiset where $\mathcal{P}=\{\langle g_1\rangle_{\fq},\ldots,\langle g_n\rangle_{\fq}\} \subseteq \PG(k-1,q)$ and with multiplicity function
\[
\mathrm{m}(P)=\lvert \{i: P=\langle g_i\rangle_{\fq}\} \rvert.
\]
We have that the codeword $vG$ has weight $w$ if and only if
the projective hyperplane
\[v_1x_1 +v_2x_2 + \cdots + v_kx_k = 0\]
contains $n  - w$ points of $(\mathcal P, \mathrm{m})$.
So, the number of distinct nonzero weights of $\C$ corresponds to the distinct sizes of the intersections of $({\mathcal P},\mathrm{m})$ with all the hyperplanes.
We refer to \cite{vladut2007algebraic} for more details on this connection.

We now give the following definition. Consider any hypercylinder $S\subseteq \PG(r,q)$. Then $S$ is a projective $[q^{r-1}+2q^{r-2},r+1,q^{r-1}]_q$-system.
A code associated with $S$ will be called a \textbf{hypercylinder code} and it is a $[q^{r-1}+2q^{r-2},r+1,q^{r-1}]_q$-code whose nonzero weights are $q^{r-1},q^{r-1}+q^{r-2}-2q^{r-3}$ and $q^{r-1}+2q^{r-2}$. In particular, these are \emph{three-weight codes} and linear codes with \emph{few weights} have been deeply investigated for their coding theoretical properties and for their applications in secret sharing \cite{carlet2005linear}, authentication codes \cite{ding2005coding}, association schemes \cite{calderbank1984three} and strongly regular graphs \cite{calderbank1986geometry}.

We can now prove a stability result on the family of hypercylinder codes.
This means that codes whose length and weight spectrum are close enough, in some sense, to those of a hypercylinder code must be hypercylinder codes.

\begin{theorem}
Let $q\geq 4$ be a prime power and let $t$ and $r\geq 3$ be positive integers such that $2q^{r-3}<t \leq q^{r-2}+q-1$. Let $\mathcal{C}$ be a projective $[q^{r-1}+q^{r-2}+t,r+1]_q$-code with nonzero weights in $\{q^{r-1}, q^{r-1}+t-2q^{r-3},q^{r-1}+q^{r-2}+t\}$. Then $q$ is even, $t=2q^{r-2}$ and $\mathcal{C}$ is a hypercylinder code.
\end{theorem}
\begin{proof}
Since $\C$ is a projective code, it is also nondegenerate and hence we can consider a projective system $S$ associated to $\C$. Moreover, $S$ turns out to be a set of $q^{r-1}+q^{r-2}+t$ points in $\PG(r,q)$ of type $\{0,q^{r-2}+2q^{r-3},q^{r-2}+t\}_{r-1}$. Then by Corollary \ref{cor:hypercylinder} we have that $q$ is even, $t=q^{r-2}$ and $S$ is a hypercylinder. Therefore, $\C$ is a hypercylinder code.
\end{proof}

Up to the action of $\mathrm{PGL}(r+1,q)$, we may always assume that the (deleted) vertex of a hypercylinder has equations $x_0=x_1=x_2=0$, and we can take the plane $x_3 = \ldots = x_r = 0$ as the plane containing the basis of the cone.
Then it is easy to see that two hypercylinders with basis the hyperovals $\mathcal{H}_1$ and $\mathcal{H}_2$, respectively, are $\mathrm{PGL}(r+1,q)$-equivalent if and only if $\mathcal{H}_1$ and $\mathcal{H}_2$ are $\mathrm{PGL}(3,q)$-equivalent.
Therefore, we have the following result.

\begin{corollary}
Two hypercylinder codes are equivalent if and only if the bases of two related hypercylinders are $\mathrm{PGL}(3,q)$-equivalent.
In particular, there is a hypercylinder code for any dimension greater than or equal to $4$.
\end{corollary}
\begin{proof}
The first part follows from the fact that the equivalence codes can be read as the $\mathrm{PGL}(r+1,q)$-equivalence of the associated hypercylinders, which is equivalent the bases being $\mathrm{PGL}(3,q)$-equivalent.
The last part follows from the fact that for any even value of $q$, there always exists at least one hyperoval in $\PG(2,q)$, namely a conic together with its nucleus.
\end{proof}

The classification of hyperovals is a quite hard and well-studied problem. A complete classification is known only in $\PG(2,q)$ when $q\leq 64$, see \cite{vandendriessche2019classification} and the references therein.

\subsection{Linear rank metric codes and systems}

Rank metric codes were introduced by Delsarte \cite{de78} in 1978 as sets of bilinear forms and they have been intensively investigated in recent years because of their several applications; see e.g.\ \cite{gorla2021rank,sheekeysurvey,polverino2020connections}.
Instead of using framework introduced by Delsarte, in this section we will be interested in rank metric codes in $\F_{q^n}^\ell$.
In this context, the \textbf{rank} (weight) $w(v)$ of a vector $v=(v_1,\ldots,v_{\ell}) \in \F_{q^n}^{\ell}$ is defined as the dimension of the vector space generated over $\F_q$ by its entries, i.e $w(v)=\dim_{\fq} (\langle v_1,\ldots, v_{\ell}\rangle_{\fq})$. 

A \textbf{(linear) rank metric code} $\C$ is an $\F_{q^n}$-subspace of $\F_{q^n}^{\ell}$ endowed with the rank distance, defined as follows
\[
d(x,y)=w(x-y),
\]
where $x, y \in \F_{q^n}^{\ell}$. 

We say that $\C$ is an $[\ell,k,d]_{q^n/q}$ code (or $[\ell,k]_{q^n/q}$ code) if $k=\dim_{\F_{q^n}}(\C)$ and $d=d(\C)$ is its minimum distance, defined analogously as for the Hamming metric. 

Moreover, we say that two rank metric codes $\C,\C' \subseteq \F_{q^n}^{\ell}$ are \textbf{(linearly) equivalent} if and only if there exists a matrix $A \in \mathrm{GL}(\ell,q)$ such that
$\C'=\C A=\{vA : v \in \C\}$. 
The codes we will consider are \textbf{nondegenerate}, that is the columns of any generator matrix of $\C$ are $\fq$-linearly independent. Denote by $ \mathfrak{C}[\ell,k,d]_{q^n/q}$ the set of all $[\ell,k,d]_{q^n/q}$ rank metric codes in $\F_{q^n}^{\ell}$.

The analogue of the projective systems in the rank metric are the systems. 
An $[\ell,k,d]_{q^n/q}$ \textbf{system} $U$ is an $\F_q$-subspace of $\F_{q^n}^k$ of dimension $\ell$, such that
$ \langle U \rangle_{\F_{q^n}}=\F_{q^n}^k$ and
$$ d=\ell-\max\left\{\dim_{\F_q}(U\cap H) \mid H \textnormal{ is an $\F_{q^n}$-hyperplane of }\F_{q^n}^k\right\}.$$
Moreover, two $[\ell,k,d]_{q^n/q}$ systems $U$ and $U'$ are \textbf{equivalent} if there exists an $\F_{q^n}$-isomorphism $\varphi\in\mathrm{GL}(k,q^n)$ such that
$$ \varphi(U) = U'.$$
We denote the set of equivalence classes of $[\ell,k,d]_{q^n/q}$ systems by $\mathfrak{U}[\ell,k,d]_{q^n/q}$.
The following result describes the relation between rank metric codes and systems.

\begin{theorem}[{\cite{Randrianarisoa2020ageometric}}] \label{th:connection}
Let $\C$ be a nondegenerate $[\ell,k,d]_{q^n/q}$ rank metric code and let $G$ be a generator matrix for $\C$.
Let $U \subseteq \F_{q^n}^k$ be the $\F_q$-span of the columns of $G$.
The rank weight of an element $x G \in \C$, with $x \in \F_{q^n}^k$ is
\begin{equation}\label{eq:relweight}
w(x G) = \ell - \dim_{\fq}(U \cap x^{\perp}),\end{equation}
where $x^{\perp}=\{y \in \F_{q^n}^k \colon x \cdot y=0\}.$ In particular,
\begin{equation} \label{eq:distancedesign}
d=\ell - \max\left\{ \dim_{\fq}(U \cap H)  \colon H\mbox{ is an } \F_{q^n}\mbox{-hyperplane of }\F_{q^n}^k  \right\}.
\end{equation}
\end{theorem}

Actually, the above result allows us to give a one-to-one correspondence between equivalence classes of nondegenerate $[\ell,k,d]_{q^n/q}$ codes and equivalence classes of $[\ell,k,d]_{q^n/q}$ systems, see \cite{Randrianarisoa2020ageometric}.
The system $U$ and the code $\C$ as in Theorem \ref{th:connection} are said to be \textbf{associated}.

Suppose that $U$ defines a cone with basis a properly maximum $h$-scattered linear set as in Construction \ref{con:genhyper}. Clearly, $U$ is a system in $\F_{q^n}^{r}$ with parameters $\left[ \frac{dn}{h+1}+n(r-d), r, \overline{d} \right]_{q^n/q}$, where 
\begin{align*}
 \overline{d}= \frac{dn}{h+1}+n(r-d) -
 \max\left\{ n(r-d)+\frac{dn}{h+1}-n+i \colon i \in \{0,\ldots,h\} \right\}
 = n-h
\end{align*}
by Corollary \ref{lem:intersectioninfinity}.
Therefore, by Theorem \ref{th:connection} a code associated with $U$ is a nondegenerate $\left[ \frac{dn}{h+1}+n(r-d), r, n-h \right]_{q^n/q}$ code whose possible weights are
\[ n-i\,\,\text{with}\,\, i \in \{0,\ldots,h\}.\]
Therefore, the maximum number of distinct weights that $\C$ may have is $2(h+1)$, and hence when $h$ is very small with respect to the length of the code $\C$ then $\C$ has \emph{few weights}.

Another class of interesting rank metric codes arise from the construction explored in Section \ref{sec:constr1}.
Indeed, let $U_1$ be as in Theorem \ref{th:intersectsizes} and let $\mathcal{C}$ be an associated code. Then $U_1$ is a system in $\fqn^{r+1}$ with parameters $\left[  \frac{dn}{h+1}+n(r-d)+1,r,1 \right]_{q^n/q}$ since
\[ d(\mathcal{C})= \frac{dn}{h+1}+n(r-d)+1 -n(r-d)+\frac{dn}{h+1}=1. \]
As a consequence of Remark \ref{rk:weighthyperconstrB} and Theorem \ref{th:connection}, the weight distribution of $\mathcal{C}$ is
 $\{ 1 \} \cup \{ n-h, n-h+1, \dots, n+1 \}$.

\section{Conclusions and open problems}\label{sec:final}

In this paper we investigated cones having as basis a properly maximum $h$-scattered linear set in a complementary subspace to the vertex. First we determined the intersection sizes of such cones with the hyperplanes. Then we analyzed two constructions of point sets both having as part at infinity a cone with basis a properly maximum $h$-scattered linear set. For both the constructions we were able to determine their intersections with the hyperplanes. 
As an instance of the second construction we obtained hypercylinders for which we provided a stability result.
Then we applied our result to codes in both the Hamming and the rank metric. Indeed, we constructed codes with few weights and we provided a stability result for the Hamming metric codes associated with hypercylinders.
We now list some open problems/questions related to the results of this paper.

\begin{itemize}
    \item It is a natural question to ask whether or not there exists a generalization of KM-arcs in projective spaces. Indeed, when considering $q=2,d=2$ and $h=1$, in Theorem \ref{th:intersectsizes}, we obtain examples of point sets $S \subseteq \PG(r,2^n)$, of size $2^{n(r-1)}+2^{n(r-2)+1}$ of type $(0,2^{n(r-2)}+2^{n(r-3)+1},2^{n(r-2)+1})_{r-1}$ and it looks exactly like what we would imagine as a generalization of KM-arcs in projective spaces. But when relaxing the conditions on the size of the point set and the intersections with the hyperplanes, we obtained again a hypercylinder. 
    \item Are there any other families of linear sets for which constructions from Sections \ref{sec:constr1} and \ref{sec:constr2} give families of point sets with few intersection numbers with respect to the hyperplanes? Note that this is a certainly non trivial task since to follow our approach we both need the weight distribution of the hyperplanes with respect to the linear set and the intersection numbers with the hyperplanes. When considering scattered linear sets we just need one of the two but, as for the cone case, in general this is not enough.
    \item We have provided a stability result for hypercylinders. It would be nice to prove that every even set in $\PG(n,q)$, $q$ even, of size $(q+2)q^{n-2}$ is a hypercylinder, or to find other examples of such sets.
    One reason why this is interesting is that these sets are equivalent to the minimum weight codewords of the binary dual code of points and lines in $\PG(n,q)$, see \cite{calkin1999minimumweight}.
\end{itemize}


\section*{Acknowledgments}

We would like to thank Jan De Beule, Sam Mattheus and Olga Polverino for fruitful discussions. The third and the fourth authors are very grateful for the hospitality of the Department of Mathematics and Data Science, Vrije Universiteit Brussel, Brussel, Belgium, where the third author was a visiting PhD student for 2 months and the fourth author was a visiting researcher for 1 month during the development of this research.
The third and the last authors were supported by the project ``VALERE: VAnviteLli pEr la RicErca" of the University of Campania ``Luigi Vanvitelli'' and by the Italian National Group for Algebraic and Geometric Structures and their Applications (GNSAGA - INdAM).

\bibliographystyle{abbrv}
\bibliography{biblio}

\begin{thebibliography}{10}

\bibitem{ball2000linear}
S.~Ball, A.~Blokhuis, and M.~Lavrauw.
\newblock Linear $(q+1)$-fold blocking sets in {PG}$(2,q^4)$.
\newblock {\em Finite Fields and Their Applications}, 6(4):294--301, 2000.

\bibitem{bartoli2021evasive}
D.~Bartoli, B.~Csajb{\'o}k, G.~Marino, and R.~Trombetti.
\newblock Evasive subspaces.
\newblock {\em Journal of Combinatorial Designs}, 2021.

\bibitem{bartoli2018maximum}
D.~Bartoli, M.~Giulietti, G.~Marino, and O.~Polverino.
\newblock Maximum scattered linear sets and complete caps in galois spaces.
\newblock {\em Combinatorica}, 38(2):255--278, 2018.

\bibitem{blokhuis2000scattered}
A.~Blokhuis and M.~Lavrauw.
\newblock Scattered spaces with respect to a spread in {PG}$(n,q)$.
\newblock {\em Geometriae Dedicata}, 81(1):231--243, 2000.

\bibitem{Brown}
M.~Brown.
\newblock (hyper)ovals and ovoids in projective spaces.
\newblock {\em Socrates Intensive Course Finite Geometry and its Applications},
  2000.

\bibitem{calderbank1984three}
A.~Calderbank and J.~Goethals.
\newblock Three-weight codes and association schemes.
\newblock {\em Philips J. Res}, 39(4-5):143--152, 1984.

\bibitem{calderbank1986geometry}
R.~Calderbank and W.~M. Kantor.
\newblock The geometry of two-weight codes.
\newblock {\em Bulletin of the London Mathematical Society}, 18(2):97--122,
  1986.

\bibitem{calkin1999minimumweight}
N.~J. Calkin, J.~D. Key, and M.~J. de~Resmini.
\newblock Minimum weight and dimension formulas for some geometric codes.
\newblock {\em Designs, Codes and Cryptography}, 17(1):105--120, 1999.

\bibitem{carlet2005linear}
C.~Carlet, C.~Ding, and J.~Yuan.
\newblock Linear codes from perfect nonlinear mappings and their secret sharing
  schemes.
\newblock {\em IEEE Transactions on Information Theory}, 51(6):2089--2102,
  2005.

\bibitem{csajbok2017maximum}
B.~Csajb{\'o}k, G.~Marino, O.~Polverino, and F.~Zullo.
\newblock Maximum scattered linear sets and {MRD}-codes.
\newblock {\em Journal of Algebraic Combinatorics}, 46(3):517--531, 2017.

\bibitem{csajbok2021generalising}
B.~Csajb{\'o}k, G.~Marino, O.~Polverino, and F.~Zullo.
\newblock Generalising the scattered property of subspaces.
\newblock {\em Combinatorica}, 41(2):237--262, 2021.

\bibitem{deboeck2016linear}
M.~De~Boeck and G.~Van~de Voorde.
\newblock A linear set view on km-arcs.
\newblock {\em Journal of Algebraic Combinatorics}, 44(1):131--164, 2016.

\bibitem{de2011constructions}
F.~De~Clerck and N.~Durante.
\newblock Constructions and characterizations of classical sets in {PG}(n,q).
\newblock {\em Current research topics in Galois geometry, Nova Sci. Publ., New
  York}, pages 1--32, 2011.

\bibitem{de78}
P.~Delsarte.
\newblock Bilinear forms over a finite field, with applications to coding
  theory.
\newblock {\em Journal of Combinatorial Theory, Series A}, 25(3):226--241,
  1978.

\bibitem{ding2005coding}
C.~Ding and X.~Wang.
\newblock A coding theory construction of new systematic authentication codes.
\newblock {\em Theoretical computer science}, 330(1):81--99, 2005.

\bibitem{gorla2021rank}
E.~Gorla.
\newblock Rank-metric codes.
\newblock In {\em Concise Encyclopedia of Coding Theory}, pages 227--250.
  Chapman and Hall/CRC, 2021.

\bibitem{hirschfeld1979}
J.~W.~P. Hirschfeld.
\newblock {\em Projective geometries over finite fields}.
\newblock Oxford Mathematical Monographs. The Clarendon Press, Oxford
  University Press, New York, 1979.

\bibitem{korchmaros1990on}
G.~Korchm{\'a}ros and F.~Mazzocca.
\newblock On {$(q+ t)$}-arcs of type {$(0, 2, t)$} in a desarguesian plane of
  order {$q$}.
\newblock In {\em Mathematical Proceedings of the Cambridge Philosophical
  Society}, volume 108, no.\ 3, pages 445--459. Cambridge University Press,
  1990.

\bibitem{lavrauw2016scattered}
M.~Lavrauw.
\newblock Scattered spaces in galois geometry.
\newblock {\em Contemporary developments in finite fields and applications},
  pages 195--216, 2016.

\bibitem{lavrauw2015field}
M.~Lavrauw and G.~Van~de Voorde.
\newblock Field reduction and linear sets in finite geometry.
\newblock In {\em Topics in finite fields}, volume 632 of {\em Contemp. Math.},
  pages 271--293. Amer. Math. Soc., Providence, RI, 2015.

\bibitem{lunardon2017mrd}
G.~Lunardon.
\newblock {MRD}-codes and linear sets.
\newblock {\em Journal of Combinatorial Theory, Series A}, 149:1--20, 2017.

\bibitem{napolitano2021linear}
V.~Napolitano, O.~Polverino, G.~Zini, and F.~Zullo.
\newblock Linear sets from projection of {D}esarguesian spreads.
\newblock {\em Finite Fields and Their Applications}, 71:101798, 2021.

\bibitem{1930-5346_2019_0_108}
V.~Napolitano and F.~Zullo.
\newblock Codes with few weights arising from linear sets.
\newblock {\em Advances in Mathematics of Communications}, 2020.

\bibitem{polverino2010linear}
O.~Polverino.
\newblock Linear sets in finite projective spaces.
\newblock {\em Discrete mathematics}, 310(22):3096--3107, 2010.

\bibitem{polverino2020connections}
O.~Polverino and F.~Zullo.
\newblock Connections between scattered linear sets and {MRD}-codes.
\newblock {\em Bulletin of the Institute of Combinatorics and its
  Applications}, 89:46--74, 2020.

\bibitem{Randrianarisoa2020ageometric}
T.~H. Randrianarisoa.
\newblock A geometric approach to rank metric codes and a classification of
  constant weight codes.
\newblock {\em Designs, Codes and Cryptography}, 88:1331–1348, 2020.

\bibitem{segre_1955}
B.~Segre.
\newblock Ovals in a finite projective plane.
\newblock {\em Canadian Journal of Mathematics}, 7:414–416, 1955.

\bibitem{sheekey2016new}
J.~Sheekey.
\newblock A new family of linear maximum rank distance codes.
\newblock {\em Advances in Mathematics of Communications}, 10(3):475, 2016.

\bibitem{sheekeysurvey}
J.~Sheekey.
\newblock {MRD} codes: constructions and connections.
\newblock {\em Combinatorics and Finite Fields: Difference Sets, Polynomials,
  Pseudorandomness and Applications}, 23, 2019.

\bibitem{sheekeyVdV}
J.~Sheekey and G.~Van~de Voorde.
\newblock Rank-metric codes, linear sets, and their duality.
\newblock {\em Designs, Codes and Cryptography}, 88:655--675, 2020.

\bibitem{vandendriessche2019classification}
P.~Vandendriessche.
\newblock Classification of the hyperovals in {PG}(2, 64).
\newblock {\em The Electronic Journal of Combinatorics}, pages P2--35, 2019.

\bibitem{vladut2007algebraic}
S.~Vladut, D.~Nogin, and M.~Tsfasman.
\newblock Algebraic geometric codes: basic notions, 2007.

\bibitem{zini2021scattered}
G.~Zini and F.~Zullo.
\newblock Scattered subspaces and related codes.
\newblock {\em Designs, Codes and Cryptography}, pages 1--21, 2021.

\end{thebibliography}

\end{document}